\documentclass[a4paper,11pt]{article}

\usepackage{a4wide} 
\usepackage{hyperref}
\hypersetup{colorlinks=true,citecolor=blue,linkcolor=blue,urlcolor=blue}
\usepackage{amsmath}
\usepackage{amssymb}
\usepackage{amsthm}   
\usepackage{enumerate}
\usepackage{thm-restate}
\usepackage{mathtools}
\usepackage{amsmath}
\usepackage{amsfonts}
\usepackage{amssymb} 
\usepackage{tikz}
\usepackage{thm-restate}

\theoremstyle{plain}
\newtheorem{theorem}{Theorem}
\newtheorem*{theorem*}{Theorem}
\newtheorem{corollary}[theorem]{Corollary}
\newtheorem{lemma}[theorem]{Lemma}
\newtheorem{observation}[theorem]{Observation}
\newtheorem{definition}[theorem]{Definition}
\newtheorem{remark}[theorem]{Remark}

\newtheorem*{lemmixtreelike}{Lemma~\ref{lem:mixtreelike}}

\usepackage[UKenglish]{babel}
\usepackage[UKenglish]{isodate}
\usepackage{url}
\usepackage{soul}
\usepackage[backgroundcolor=green]{todonotes}

\let\epsilon=\varepsilon

\newcommand{\N}{\mathbb{N}}

\newcommand{\Tmix}{T_\mathrm{mix}}
\newcommand{\G}{\mathcal{G}}

\newcommand{\C}{\mathcal{C}}

\newcommand{\B}{\mathcal{B}}
\newcommand{\A}{\mathcal{A}}
\newcommand{\emm}{\mathrm{e}}
\newcommand{\Ec}{\mathcal{E}}
\newcommand{\Pc}{\mathcal{P}}
\newcommand{\Xc}{\mathcal{X}}
\newcommand{\Eb}{\mathbf{E}}
\newcommand{\Ent}{\mathrm{Ent}}
\newcommand{\F}{\mathcal{F}}
\newcommand{\dist}{\mathrm{dist}}
\newcommand{\TV}{\mathrm{TV}}

\newcommand{\M}{M}

\newcommand{\tv}{\mathrm{TV}}
\DeclareMathOperator{\In}{\mathrm{In}}
\DeclareMathOperator{\Out}{\mathrm{Out}}
 
\newcommand{\ord}{\mathrm{ord}}
\newcommand{\dis}{\mathrm{dis}}

\title{Sampling from the random cluster model on random regular graphs at all temperatures via Glauber dynamics\thanks{A preliminary version of the paper (without the full proofs) appeared in the proceedings of RANDOM 2023. For the purpose of Open Access, the authors have applied a CC BY public copyright licence to any Author Accepted Manuscript version arising from this submission. All data is provided in full in the results section of this paper.}}
\author{Andreas Galanis \and Leslie Ann Goldberg \and Paulina Smolarova}

\date{13 September 2023}

\begin{document}

\maketitle
\begin{abstract}
We consider the performance of Glauber dynamics for the random cluster model with real parameter $q>1$ and temperature $\beta>0$.  Recent work by Helmuth, Jenssen and Perkins detailed the ordered/disordered transition of the model on random $\Delta$-regular graphs for all sufficiently large $q$ and obtained an efficient sampling algorithm for all temperatures $\beta$ using cluster expansion methods. 
Despite this major progress, the performance of natural Markov chains, including Glauber dynamics, is not yet well understood on the random regular graph,  partly because of the non-local nature of the model (especially at low temperatures) and partly because of severe bottleneck phenomena that emerge in a window around the ordered/disordered transition.

 Nevertheless, it is widely conjectured that the bottleneck phenomena that impede mixing from worst-case starting configurations can be avoided by initialising the chain more judiciously. Our main result establishes this conjecture for  all sufficiently large~$q$ (with respect to $\Delta$).  Specifically, we consider the mixing time of Glauber dynamics initialised from the two extreme configurations, the all-in and all-out, and obtain a pair of fast mixing bounds which cover \emph{all} temperatures $\beta$, including in particular the bottleneck window.  Our result is inspired by the recent approach of Gheissari and Sinclair for the Ising model who obtained a similar-flavoured mixing-time bound on the random regular graph for sufficiently low temperatures. To cover all temperatures in the RC model, we refine appropriately the structural results of Helmuth, Jenssen and Perkins about the ordered/disordered transition and show spatial mixing properties ``within the phase'', which are then related to the evolution of the chain.
\end{abstract}

\section{Introduction}

Glauber dynamics is a well-studied Markov chain that is widely used to sample from spin systems such as the Ising and Potts models. A particularly appealing feature of the chain is that it is usually very simple to implement; yet, establishing whether convergence to the equilibrium distribution is fast turns out to be significantly harder, and typically requires an in-depth understanding of the underlying distribution. Here, we focus on studying the performance of Glauber dynamics for the random-cluster representation of the classical Potts model on the random regular graph, where the behaviour of the chain  is underpinned by various phenomena that distinguish it from classical spin systems and pose new challenges in the analysis.

We begin with a few definitions. For real numbers $q, \beta > 0$ and a graph $G = (V,E)$, the \textit{random cluster model} on $G$ with parameters $q$ and $\beta$ is a probability distribution on the set $\Omega=\Omega_G$ of all assignments $\F:E\rightarrow\{0,1\}$; we typically refer to assignments in $\Omega$ as configurations. For a configuration $\F$, we say that edges mapped to $1$ are \textit{in-edges}, and edges mapped to $0$ are \textit{out-edges}. We use $\In(\F)$ to denote the set of edges $e$ with $\F(e)=1$,  
$\Out(\F)$ to denote the set of edges~$e$ with $\F(e)=0$,
$|\F|$ for the cardinality of $\In(\F)$ and $c(\F)$ for the number of connected components in the graph $(V,\In(\F))$. Then, the weight of $\F$  in the RC model is given by $w_G(\F)= q^{c(\F)} (e^\beta-1)^{|\F|}$.

For integer values of $q$, the RC model is closely connected to the (ferromagnetic) Ising/Potts models; $q=2$ is the Ising model and $q\geq 3$ is the Potts model whose configurations are all possible assignments  of $q$ colours to the vertices of the graph where an assignment $\sigma$ has weight proportional to $\emm^{\beta m(\sigma)}$ with $m(\sigma)$ being the number of monochromatic edges under $\sigma$. The RC model is an alternative edge representation of the models (for integer $q$) that has also been studied extensively in its own right due to its intricate behaviour (see, e.g., \cite{GrimmettGeoffrey2006TrmRCM}).

We will be primarily interested in sampling from the so-called  \textit{Gibbs distribution}  on $\Omega$ induced by these weights, denoted by $\pi_G(\cdot)$, where for a configuration $\F$, 
 $\pi_G(\F) =w_G(\F)/Z_G$
where the normalising factor $Z_G = \sum_{\F'\in\Omega_G}w_G(\F')$ is the aggregate sum of weight of all configurations (known as the partition function). We focus on the \textit{Glauber dynamics} which is a classical Markov chain for sampling from Gibbs distributions which is a particularly useful tool for developing approximate sampling algorithms. We will refer to Glauber dynamics for the RC model as the RC dynamics. Roughly, the RC dynamics is a Markov chain $(X_t)_{t\geq 0}$ initialised at some configuration $X_0$ which evolves by iteratively updating at each step $t\geq 1$ a randomly chosen edge  based on whether its endpoints belong to the same component in the graph $(V,\In(X_t))$. The mixing time of the chain is the number of steps to get within total variation distance $\leq 1/4$ from $\pi_G$,  see Section~\ref{sec:prelims} for details.

Our goal is to obtain a fast algorithm for the RC model using Glauber dynamics on the random regular graph. There are two key obstacles that arise, especially at low temperatures (large $\beta$): (i) Glauber dynamics for the RC model has a  non-local behaviour since its updates depend on the component structure of the running configuration, and (ii) there are severe bottleneck phenomena and worst case graphs which prohibit a general fast-convergence result, and more generally an efficient algorithm. The random regular graph is a particularly interesting testbed in this front since it exhibits all the relevant phase transition phenomena and has also been used as the main gadget in hardness reductions \cite{GSVY}.

To overview the phenomena that are most relevant for us, the following picture was detailed in a remarkable development by Jenssen, Helmuth, and Perkins \cite{RCM-Helmuth2020}: for $\Delta\geq 5$ and all sufficiently large $q$,  they established the ordered/disordered transition occurring at some $\beta_c$ satisfying $\beta_c=(1+o_q(1))\tfrac{2\log q}{\Delta}$ (see also \cite{GSVY} for integer $q\geq 3$).\footnote{Recent results of Bencs, Borb{\'e}nyi, and   Csikv{\'a}ri \cite{bencs2022random} yield the exact formula $\beta_c=\log \tfrac{q-2}{(q-1)^{1-2/\Delta}-1}$ for all $q>2$ and $\Delta\geq 3$, which was previously only known for integer $q$ \cite{GSVY}.} Roughly, for $\beta<\beta_c$ a typical configuration of the model is disordered, whereas for $\beta>\beta_c$ it is ordered: disordered configurations resemble the all-out configuration (in that all components are of size $O(\log n)$) whereas ordered configurations resemble the all-in configuration (where there is a giant component with $\Omega(n)$ vertices). The two types of configurations coexist at $\beta=\beta_c$, i.e., each appears with some probability bounded away from zero. The methods in \cite{RCM-Helmuth2020} are based on cluster expansion techniques which also yielded an efficient sampling algorithm at all temperatures $\beta>0$. This is a surprising algorithmic result given that the coexistence causes  multimodality in $\pi_G$ and severe bottleneck phenomena for Markov chains in a window around $\beta_c$; it was shown for instance in \cite{RCM-Helmuth2020} that the RC dynamics (and the related non-local Swendsen-Wang dynamics) have exponential mixing time, essentially because of the number of steps needed for the chain to move from ordered to disordered (and vice versa). 

These results pose a rather bleak landscape for the RC dynamics;  yet, on random regular graphs it is widely conjectured that the multimodality and the associated bottlenecks can be circumvented by initialising the chain more judiciously, in particular at either the all-out or the all-in configurations (depending on whether $\beta\leq\beta_c$).  However the tools available for analysing Markov chains are typically insensitive to the initial configuration, and even more so when working at a critical range of the parameters.

Our main result establishes this conjecture for all $\Delta\geq 5$ and $q$ sufficiently large (conditions which we inherit from \cite{RCM-Helmuth2020}). For an integer $n$ such that $\Delta n$ is even, let $\G_{n,\Delta}$ denote the set of all $\Delta$-regular graphs with $n$ vertices.\footnote{We write $G\sim \G_{n,\Delta}$ to denote a graph in $\G_{n,\Delta}$ chosen uniformly at random, and we say that a property holds w.h.p. for $G\sim \G_{n,\Delta}$  as a shorthand for ``with probability $1-o_n(1)$ over a graph  $G\in \G_{n,\Delta}$ chosen uniformly at random.} Throughout, we use $O(1)$ to denote a constant depending on $q,\beta,\Delta$ but independent of $n$.
\newcommand{\statethmmain}{Let $\Delta\geq 5$ be an integer. There exists $C=C(\Delta)>0$ such that, for all sufficiently large $q$,  the following holds for any $\beta>0$, w.h.p. over $G\sim \G_{n,\Delta}$. 
\begin{enumerate}
    \item For $\beta<\beta_c$,  the mixing time of the RC dynamics starting from all-out  is $O(n\log n).$
    \item For $\beta> \beta_c$, the mixing time of the RC dynamics starting from all-in is $O(n^{C})$. For integer $q$, the mixing time is in fact $O(n\log n)$.
\end{enumerate}}
\begin{theorem}\label{thm:main1}
\statethmmain
\end{theorem}

Note that Theorem~\ref{thm:main1} implies 
a polynomial-time 
sampling algorithm from the Potts model for all $\beta\neq \beta_c$ (and all sufficiently large $q$). Intuitively, and as we will see later in more detail, Theorem~\ref{thm:main1} asserts that the RC dynamics starting from all-in mixes quickly within the set of  ordered configurations for $\beta> \beta_c$, and similarly it mixes well within the disordered set of configurations starting from all out when $\beta<\beta_c$. In fact, we show that both of these remain true even for  $\beta=\beta_c$ and hence the RC dynamics can be used to sample even at criticality, starting from an appropriate mixture of the all-in and the all-out configurations. See Section~\ref{sec:prelims} (and Section~\ref{sec:f3434f34} in particular) for the exact statements.

Finally, note that the RC dynamics  can be used analogously to  Theorem~\ref{thm:main1} to produce a sample within total variation distance $\epsilon$ of $\pi_G$ for any $\epsilon\geq\emm^{-\Theta(n)}$, by running it for 
a number of steps which is
$\log(1/\epsilon)$ times the corresponding mixing time bound.\footnote{The standard submultiplicative argument to bootstrap the total-variation distance  goes through using the monotonicity of the RC model (to account for  the constraint on the initial configuration), see also \cite{peres2013}.} The lower bound on the error comes from the total variation distance between $\pi_G$ and the conditional ``ordered'' and ``disordered'' configurations, see Lemma~\ref{lem:dominant}.

\subsection{Further related work} 

Our approach to proving Theorem~\ref{thm:main1} is inspired from a recent paper by Gheissari and Sinclair~\cite{SinclairsGheissari2022} who established similar flavoured results for the Ising model ($q=2$) on the random regular graph for large $\beta$.  To obtain our results for all $\beta$, we adapt suitably their notion of  ``spatial mixing within the phase'', see Section~\ref{sec:ingWSM} for details. 

 Among the results in  \cite{SinclairsGheissari2022}, it was established that Glauber dynamics on the  random regular graph, initialised appropriately, mixes in $O(n \log n)$ time when $\beta$ is sufficiently large.\footnote{Note that, for $q=2$, an $O(n^{10})$ upper bound for the RC dynamics on any graph $G$ was previously known at all temperatures $\beta$ by Guo and Jerrum \cite{GuoJerrum} (see also \cite{FGW}).} More  recently, Gheissari and Sinclair \cite{RClattice} obtained mixing-time bounds for the RC dynamics on the lattice $\mathbb{Z}^d$  under appropriate boundary conditions. They also analyse the mixing time starting from a mixture of the all-in/all-out initialisation. Note that the phase transition on grid lattices is qualitatively different than that of the random regular graph; there, instead of a window/interval of temperatures, the three points $\beta_u,\beta_u^*$ and $\beta_c$ all coincide into a single phase transition point.  See also \cite{borgs2020efficient, helmuth2019algorithmic} for related algorithmic results on $\mathbb{Z}^d$ using cluster expansion methods.

For the random regular graph, Blanca and Gheissari \cite{blanca2021random} showed for all integer $\Delta\geq 3$ and real $q\geq 1$ that the mixing time is $O(n\log n)$ provided that $\beta<\beta_u(q,\Delta)$ where $\beta_u$ is  the  uniqueness threshold on the tree. A sampling algorithm (not based on MCMC)   for  $\beta<\beta_c(q,\Delta)$  and $q,\Delta\geq 3$ was designed by Efthymiou~\cite{Charilaos} (see also \cite{Blanca2}), albeit achieving weaker approximation guarantees. Coja-Oghlan et al.~\cite{coja2023} showed that, for all integer $q,\Delta\geq 3$ and $\beta\in (\beta_u,\beta_u')$ the mixing time is $\emm^{\Omega(n)}$ where $\beta_u'=\log(1+\tfrac{q}{\Delta-2})>\beta_u$ is (conjectured to be) another uniqueness threshold on the tree (see \cite{haggstrom1996random,JONASSON}). More generally, for integer $q\geq 3$, the hardness results/techniques of \cite{GJ,GSVY} yield that for any $\beta>\beta_c$, there are graphs $G$ where the mixing time of the RC dynamics is $\exp(n^{\Omega(1)})$ and the problem of appoximately sampling on graphs of max-degree $\Delta$ becomes \#BIS-hard; on the other hand, for $\beta\leq (1-o_{q,\Delta}(1))\beta_c$ it has been shown in \cite{Coulson,carlson2022algorithms} that the cluster-expansion technique of \cite{RCM-Helmuth2020} yields a sampling algorithm on any max-degree $\Delta$ graph.

As a final note, another model of interest where analogous mixing results for Glauber dynamics (initialised appropriately) should be obtainable is for sampling independent sets on random bipartite regular graphs. However, in contrast to the RC/Potts models, the phase transition there is analogous to that of the Ising model, and hence, establishing the relevant spatial mixing properties close to the criticality threshold is likely to require different techniques, see, e.g., \cite{optimalBIS} for more discussion.

In the next section, we outline the proof of Theorem~\ref{thm:main1}, explaining the main ingredients and showing how to combine them in order to conclude the proof. The rest of the paper is about establishing the main ingredients, see also Section~\ref{sec:dwefwe} for a more detailed overview of the later parts.

\subsection{Independent results of Blanca and Gheissari}
In an independent and simultaneous work, Blanca and Gheissari \cite{BlaGhe} obtain related (but incomparable) results. For $\Delta\geq 3, q\geq 1$ and arbitrarily small $\tau>0$, they show for sufficiently large $\beta$ a mixing time bound of $O(n^{1+\tau})$ for the RC dynamics on the random regular graph starting from an arbitrary configuration (and obtain an analogous result for the grid and the Swendsen-Wang dynamics). Our result instead applies to all $\beta$ for the random regular graph (even the critical window) by taking into consideration the initial configuration; the two papers have different approaches to obtain the main ingredients.

\section{Proof of Theorem~\ref{thm:main1} }\label{sec:prelims}
We start with the formal description of the RC dynamics. Given a graph $G=(V,E)$ and an initial configuration $X_0:E\rightarrow \{0,1\}$, the RC dynamics on $G$ is a Markov chain $(X_t)_{t\geq 0}$ on the set of configurations $\Omega_G$. Let $p:=1-\emm^{-\beta}$ and $\hat{p}:=\tfrac{p}{(1-p)q+p}$ (note that for $q>1$ it holds that $\hat p\in(p/q,p)$). For $t\geq 0$, to obtain $X_{t+1}$ from $X_t$:
\begin{enumerate}
\item Choose u.a.r. an edge $e\in E$. If $e$ is a cut-edge in the graph $(V,\In(X_t)\cup\{e\})$, set $X_{t+1}(e)=1$ with probability $\hat{p}$ (and $X_{t+1}(e)=0$ otherwise).  Else, set $X_{t+1}(e)=1$ with probability $p$, and $X_{t+1}(e)=0$ otherwise.
\item Set $X_{t+1}(f)=X_t(f)$ for all $f\in E\backslash \{e\}$.
\end{enumerate}
It is a standard fact that the distribution of $X_t$ converges to the RC distribution $\pi_G$. Let $\Tmix(G;X_0)=\min_{t\geq 0} \{t\mid \dist_{\mathrm{TV}}(X_t,\pi_G)\leq 1/4\}$ be the number of steps needed to get within total-variation distance $\leq 1/4$ from $\pi_G$ starting from $X_0$, and $\Tmix(G)=\max_{X_0}\Tmix(G;X_0)$ be the mixing time from the worst starting state.

\subsection{The ordered and disordered phases on random regular graphs}\label{sec:orddis}
We review in more detail the ordered/disordered transition, following~\cite{RCM-Helmuth2020}.

\begin{definition}\label{def:ord}
For $\Delta\geq 3$, let $\eta=\eta(\Delta)\in (0,1/2)$ be a small constant  (see Definition~\ref{def:eta}). For $G\in \G_{n,\Delta}$,    the \emph{ordered phase}  is the set of configurations  
$ \Omega^\ord\coloneqq \{\F\in\Omega : |\In(\F)|\geq (1-\eta)|{E}|\}$, whereas the \emph{disordered phase} is the set  $ \Omega^\dis\coloneqq \{\F\in\Omega : |\In(\F)|\leq \eta|{E}| \}$. 
For $q,\beta>0$, let $\pi^\ord_G,\pi^\dis_G$ be the conditional distributions of $\pi_G$ on $\Omega^\ord,\Omega^\dis$, respectively.  
\end{definition}

We will use the following result of Helmuth, Jenssen and Perkins   \cite[Lemma 9]{RCM-Helmuth2020}.  

\begin{lemma}[{\cite[Theorem 1]{RCM-Helmuth2020}}]\label{lem:dominant}
Let $\Delta\geq 5$ be an integer. Then, for all sufficiently large $q$, there exists $\beta_c>0$  satisfying $\beta_c=(1+o_q(1))\tfrac{2\log q}{\Delta}$ such that the following holds for any $\beta>0$ w.h.p. for $G\sim \G_{n,\Delta}$. 
\begin{equation}\label{eq:4tt}
\mbox{if $\beta<\beta_c$, then $\left\| \pi_G-\pi_G^\dis\right\|_{\TV}=\emm^{-\Omega(n)}$}; \qquad \mbox{if $\beta>\beta_c$, then $\left\| \pi_G-\pi_G^\ord\right\|_{\TV}=\emm^{-\Omega(n)}$}.
\end{equation}
Moreover, there exists $\zeta=\zeta(\Delta)>0$ with $\zeta<\eta$ such that  
\begin{equation}\label{eq:margin}
\begin{aligned}
\mbox{for $\beta\leq \beta_c$}, &\quad  \pi_G^\dis\Big(|\In(\F)|\geq \zeta|E|\Big)=\emm^{-\Omega(n)}, \quad \mbox{ and }\\ 
\mbox{for $\beta\geq \beta_c$}, &\quad  \pi_G^\ord\Big(|\In(\F)|\leq (1-\zeta)|E|\Big)=\emm^{-\Omega(n)}.
\end{aligned}
\end{equation}
\end{lemma}
\begin{proof}
The claims about the total variation distance are shown in {\cite[Theorem 1, Items (2), (3), (8)]{RCM-Helmuth2020}}. Equation~\eqref{eq:margin} shows a bit of slack in the definitions of $\Omega^{\dis}$ and $\Omega^{\ord}$ that will be useful later; it follows essentially from the same theorem, we defer the details to Lemma~\ref{lem:slack}.
\end{proof}

\subsection{Main ingredient: Weak spatial mixing within a phase}\label{sec:ingWSM}
Let $G=(V,E)$ be a graph. For $v\in V$ and $r\geq 0$, let $B_r(v)$ denote the set of all vertices in~$V$ whose distance from $v$ is at most $r$.  Let $\pi=\pi_G$ be the RC distribution on $G$ and let $\pi_{\B_r^+(v)}$ be the conditional distribution of $\pi$ where all edges in $E\backslash E(B_r(v))$ are ``in''. We define analogously $\pi_{\B_r^-(v)}$ by conditioning the edges in $E\backslash E(B_r(v))$ to be ``out''.

\begin{definition}\label{def:WSM}
Let $G$  be a graph with $m$ edges. Let $q,\beta>0$ be reals and $r\geq 1$ be an integer. 
We say that the graph $G$ has  \emph{WSM within the ordered phase at radius $r$} if for every $v\in V(G)$ and every edge $e$ incident to $v$,  $\Vert{\pi_{\B_r^+(v)}(e \mapsto \cdot ) - \pi^{\ord}_G(e\mapsto \cdot) }\Vert_{\tv} \leq \frac{1}{100m}$. 
Analogously, we say that $G$  has \emph{WSM within the disordered phase at radius $r$} if for every $v\in V(G)$ and every edge $e$ incident to $v$, $\Vert{\pi_{\B_r^-(v)}(e \mapsto \cdot ) - \pi^{\dis}_G(e\mapsto \cdot) }\Vert_{\tv} \leq \frac{1}{100m}$.
\end{definition}

The bulk of our arguments consists of showing the following two theorems.

\begin{restatable}{theorem}{rrwsm}\label{thm:RRwsm}\label{thm:ordered}
Let $\Delta\geq 5$ be an integer. There exists $M=M(\Delta)>0$ such that for all $q$ sufficiently large, 
the following holds for any $\beta\geq \beta_c$.  W.h.p. over $G\sim \G_{n,\Delta}$,  $G$ has WSM within the ordered phase at a radius $r$ which satisfies $r\leq \tfrac{M}{\beta}\log n$.
\end{restatable}
The upper bound on the radius $r$ in terms of $1/\beta$ ensures that we can remove the dependence on $\beta$ of the mixing time  in Theorem~\ref{thm:main1} (caused by a loose bound on the mixing time on the tree, see Lemma~\ref{lem:mixtreelike} below).  For the disordered phase, we have

\begin{restatable}{theorem}{rrwsmb}\label{thm:RRwsmB}
For all integer $\Delta\geq 5$, for all $q$ sufficiently large and any $\beta \leq \beta_c$, w.h.p.  
over $G\sim \G_{n,\Delta}$,  $G$ has WSM within the disordered phase at a radius $r$ which satisfies $r\leq \tfrac{1}{3}\log_{\Delta-1}n$.
\end{restatable}

\subsection{Second ingredient: Local mixing on tree-like neighbourhoods}\label{sec:localmixing}
We first define a local version of RC dynamics where we perform only updates in a small ball around a vertex. Here, we need to consider the extreme boundary conditions that  all vertices outside of the ball belong in distinct components (``free boundary'') and where they belong to the same component (``wired boundary''); we will refer to these two chains as the free and wired RC dynamics, respectively. For the random regular graph, these ``local-mixing'' considerations are strongly connnected to the $\Delta$-regular tree.

Formally, given a graph $G=(V,E)$ and a subset $U\subseteq V$, let $G[U]$ be the induced subgraph of $G$ on $U$. The tree-excess of a connected graph $G$ is given by $|E|-|V|+1$. For a vertex $v$ in $G$ and integer $r\geq 0$, let $B_r(v)$ denote the set of vertices at distance at most $r$ from $v$ and $S_r(v)$ those at distance exactly $r$ from $v$. For $K>0$, a max-degree $\Delta$ graph $G$ is locally $K$-treelike if for every  $v\in V$ and $r\leq \frac{1}{3}\log_{\Delta-1}|V|$, the graph $G[B_r(v)]$  has tree excess $\leq K$.

\begin{lemma}[see, e.g., {\cite[Lemma 5.8]{SinclairsGheissari2022},\cite[Fact 2.3]{blanca2021random}} ]\label{lem:treelike}
For any integer $\Delta \geq 3$, there is $K>0$ such that w.h.p. $G\sim G_{n,\Delta}$ is locally $K$-treelike.
\end{lemma}

For a graph $G$, a vertex $\rho$ in $G$ and an integer $r\geq 1$, the \emph{free RC dynamics} on $B_r(\rho)$ is the RC dynamics where all edges outside of $B_r(\rho)$ are conditioned to be out  and   only edges of $G$ with  both endpoints in $B_r(\rho)$ are updated. 
\newcommand{\statelemmixtreelikeb}{Let $\Delta\geq 3$ be an integer, and  $q,K>1$, $\beta>0$ be reals. There exists  $C>0$ such that the following holds for any $\Delta$-regular graph $G$ and integer $r\geq 1$.

Suppose that $\rho\in V$ is such that $G[B_r(\rho)]$ is $K$-treelike. Then, with $n=|B_r(\rho)|$,  the mixing time of the free RC dynamics on $B_r(\rho)$ is $\leq  Cn\log n$. }
\begin{lemma}[{\cite[Lemma 6.5] {blanca2021random}}]\label{lem:mixtreelikeb}
\statelemmixtreelikeb
\end{lemma} 

To define the wired RC dynamics, for a graph $G$, a vertex $\rho$ in $G$ and an integer $r\geq 1$, let $H$ be the graph obtained by removing all vertices and edges outside of $B_r(\rho)$, and adding a new vertex $v_{\infty}$ connected to all vertices in $S_r(\rho)$. The \emph{wired RC dynamics} on $B_r(\rho)$ is the RC dynamics on $H$  where the edges adjacent to $v_\infty$ are conditioned to be in and only edges of $G$ with  both endpoints in $B_r(\rho)$ are updated. Denote by $\hat{\pi}_{B_r(\rho)}$ the stationary distribution of the wired RC dynamics.  Note that when the graph outside of $B_r(\rho)$ is connected, $\hat{\pi}_{B_r(\rho)}$ induces the same distribution as $\pi_{\B^+_r(\rho)}$.\footnote{\label{fn:chat} More precisely, the weight of a configuration $\F:E(B_r(\rho))\rightarrow \{0,1\}$ in $\hat{\pi}_{B_r(\rho)}$ is proportional to $q^{\hat{c}(\F)}(\emm^{\beta}-1)^{|\F|}$ where $\hat{c}(\F)$ denotes the number of components in the graph $(B_r(\rho),\In(F))$ that do not include any of the vertices in $S_r(\rho)$ (since all of these belong to the same component in the wired dynamics and hence contribute just a single extra factor of $q$).} 

\newcommand{\statelemmixtreelike}{Let $\Delta\geq 3$ be an integer, and  $q,K>1$, $\beta>0$ be reals. There exists  $\hat{C}>0$ such that the following holds for every $\Delta$-regular graph $G=(V,E)$ and any integer $r\geq 1$.

Suppose that $\rho\in V$ is such that $G[B_r(\rho)]$ is $K$-treelike. Then, with $n=|B_r(\rho)|$,  the mixing time of the wired RC dynamics on $B_r(\rho)$ is $\leq \hat{C}n^3(q^{4}\emm^{\beta})^{\Delta r}$. }
\begin{lemma}\label{lem:mixtreelike}
\statelemmixtreelike
\end{lemma}
Lemma~\ref{lem:mixtreelike} is proved in Appendix~\ref{app:mixtreelike}.
\begin{remark}\label{rem:treeint} For integer $q>1$, the mixing time bound in Lemma~\ref{lem:mixtreelike} can be improved to $O(n\log n)$ using  results of \cite{treemixingRC}, see Section~\ref{sec:intq} for details.
\end{remark}

\subsection{Proof of Theorem~\ref{thm:main1}}\label{sec:f3434f34}
Assuming for now the above ingredients, we will conclude the proof of Theorem~\ref{thm:main1}. For clarity, we break up the latter into the following pieces. 

\begin{theorem}[Convergence to ordered  for real $q$]\label{thm:main1b}
Let $\Delta\geq 5$ be an integer. Then, for all sufficiently large $q$, there exists $C=C(\Delta)$ such that the following holds for $\beta\geq  \beta_c$, w.h.p. over $G\sim \G_{n,\Delta}$. The RC dynamics $(X_t)_{t\geq 0}$ on $G$ starting from all-in satisfies $\dist_{\mathrm{TV}}(X_T,\pi_G^\ord)\leq 1/5$ for $T=O(n^{C})$. 
\end{theorem}
\begin{theorem}[Convergence to disordered  for real $q$]\label{thm:main1c}
Let $\Delta\geq 5$ be an integer. Then, for all sufficiently large $q$ and $\beta\leq \beta_c$, w.h.p. over $G\sim \G_{n,\Delta}$, the RC dynamics $(X_t)_{t\geq 0}$ on $G$ starting from all-in satisfies $\dist_{\mathrm{TV}}(X_T,\pi^\dis_G)\leq 1/5$ for $T=O(n\log n)$. 
\end{theorem}
\begin{theorem}[Faster convergence to ordered for integer $q$]\label{thm:main1d}
Let $\Delta\geq 5$ be an integer. Then, for all sufficiently large integer $q$ and $\beta\geq \beta_c$, w.h.p. over $G\sim \G_{n,\Delta}$, the RC dynamics $(X_t)_{t\geq 0}$ on $G$ starting from all-in satisfies $\dist_{\mathrm{TV}}(X_T,\pi^\dis_G)\leq 1/5$ for $T=O(n\log n)$.  
\end{theorem}
Using these, the proof of Theorem~\ref{thm:main1} can be concluded easily.
\begin{proof}[Proof of Theorem~\ref{thm:main1}]
Follows immediately by Theorems~\ref{thm:main1b}-\ref{thm:main1d}, since by Lemma~\ref{lem:dominant} we have that 
$\left\| \pi_G-\pi_G^\ord\right\|_{\TV}=\emm^{-\Omega(n)}$ when $\beta>\beta_c$ and $\left\| \pi_G-\pi_G^\dis\right\|_{\TV}=\emm^{-\Omega(n)}$.
\end{proof}

We next show how to combine the ingredients of Sections~\ref{sec:orddis}-\ref{sec:localmixing} in order to conclude Theorem~\ref{thm:main1b}. The proofs of Theorems~\ref{thm:main1c} and~\ref{thm:main1d} are very similar, but involve one more technical tool in order to conclude the $O(n\log n)$ bounds. We  defer their proof to Section~\ref{sec:thmmain2}.

\begin{proof}[Proof of Theorem~\ref{thm:main1b}] 
The argument resembles that of~\cite{SinclairsGheissari2022}, a bit of care is required to combine the pieces. Consider $G=(V,E)\sim \G_{n,\Delta}$ with $n=|V|$ and $m=|E|$.  Let $q$ be sufficiently large so that both Lemma~\ref{lem:dominant} and Theorem~\ref{thm:RRwsm} apply; assume also that Lemma~\ref{lem:treelike} applies so that $G$ is locally $K$-treelike. 
By Lemma~\ref{lem:dominant}, for every $\beta \geq \beta_c$,  
$\pi_G^\ord\big(|\In(\F)|\leq (1-\zeta)|E|\big)=\emm^{-\Omega(n)}$.

From now on consider 
arbitrary $\beta\geq  \beta_c$ and set $\beta_0:=\log(q^{1.9/\Delta}+1)$. 
Since $\beta_c=(1+o_q(1))\tfrac{2\log q}{\Delta}$, we have that $\beta\geq \beta_0$ for all sufficiently large $q$. 
Moreover, by Theorem~\ref{thm:RRwsm}, $G$ has  WSM within the ordered phase at  radius $r$ for some  $r\leq \tfrac{M}{\beta}\log n$, where  $M=M(\Delta)>0$ is a constant independent of $\beta$.  Note that by taking $q$ large,  we can ensure that $\beta_0$ and hence $\beta$ are at least $3M\log(\Delta-1)$ so that $r\leq \tfrac{1}{3}\log_{\Delta-1}n$ (and hence the radius-$r$ neighbourhood of an arbitrary vertex in $G$ is locally $K$-treelike).

We will consider the RC dynamics $(X_t)_{t\geq 0}$ with $X_0$ being the all-in configuration on the edges. 
We will also consider the ``ordered'' RC dynamics $\hat{X}_t$ with $\hat{X}_0\sim\pi^\ord_G$ where we reject moves that lead to configurations outside of $\Omega^{\ord}$;  since $\hat{X}_0$ is stationary note that $\hat{X}_t\sim\pi^\ord_G$ for all $t\geq 0$.
We will show how to couple these two dynamics in $O(n^C)$ steps. The theorem will follow by showing that ${\mathrm{dist}}_{\mathrm{TV}}(X_T,\hat{X}_T)\leq 1/5$, where $T=O(n^{2+\log W})$ with $W=\Delta^{2M/\beta_0} \emm^{M \Delta(\Delta+1)}$ being independent of~$\beta$.

Consider the dynamics  $(X_t)_{t\geq 0}$  and $(\hat{X}_t)_{t\geq 0}$.    For $t\geq 0$, let $\Ec_t$ be the event that $\In(\hat{X}_t)\geq (1-\zeta)|E|$ and  let $\Ec_{< t}:=\bigcap_{t'=0,\hdots,t-1} \Ec_{t'}$. From Lemma~\ref{lem:dominant} we have that $\pi^{\ord}_G(\Ec_t)\geq 1-\emm^{-\Omega(n)}$ and hence by a union bound $\pi^{\ord}_G(\Ec_{< t})\geq 1-t\emm^{-\Omega(n)}$ as well. 

We couple the  evolution of $X_t$ and $\hat{X}_t$ using the monotone coupling, i.e., at every step of the two chains choose the same edge $e_t$ to update and use the same uniform number $U_t\in [0,1]$ to decide whether to include $e_t$ in each of $X_{t+1},\hat{X}_{t+1}$. Using the monotonicity of the model for $q\geq 1$ (and in particular that $p>\hat{p}$), under the monotone coupling, for all $t\geq 0$ such that  $\Ec_{<t}$ holds (and hence no reject move has happened in $\hat{X}_t$ so far), we have that $\hat{X}_t\leq X_t$  (i.e., $\In(\hat{X}_t)\subseteq \In(X_t)$). To complete the proof, it therefore suffices to show that  
\begin{equation}\label{eq:4r34mm}
    \Pr(X_T\neq \hat{X}_T)\leq 1/4.
\end{equation}
Consider an arbitrary time $t\geq 0$. By a union bound, we have that
\begin{equation}\label{eq:rvffv3}
\Pr\big(X_{t}\neq \hat{X}_t\big)\leq \sum_{e}\Pr\big(X_t(e)\neq \hat{X}_t(e)\big)\leq m\Pr\big(\overline{ \Ec_{<t}}\big)+\sum_{e}\Pr\big(X_t(e)\neq \hat{X}_t(e)\mid \Ec_{< t}\big).
\end{equation}
Fix an arbitrary edge $e$ incident to some vertex $v$, and let $(X^v_t)$ be the wired RC dynamics on $G[B_r(v)]$. We couple the evolution of $(X_t^v)$ with that of $(X_t)$ and $(\hat{X}_t)$ using the monotone coupling analogously to above, where in $X_t^v$ we ignore updates of edges outside the ball $G[B_r(v)]$). We have  $X_t^v\geq X_t$ for all $t\geq 0$, and hence, conditioned on $\Ec_{<t}$, we have that $X^v_t\geq X_t\geq \hat{X}_t$. It follows that
\begin{align*}
\Pr\big(X_t(e)\neq \hat{X}_t(e)\mid \Ec_{< t}\big)&=\Pr(X_t(e)=1\mid \Ec_{<t})-\Pr(\hat{X}_t(e)=1\mid \Ec_{<t})\\
&\leq|\Pr(X^v_t(e)=1\mid \Ec_{<t})-\Pr(\hat{X}_t(e)=1\mid \Ec_{<t})|.
\end{align*}
For any two events $A,B$, we have $|\Pr(A)-\Pr(A\mid B)|\leq 2 \Pr(\overline{B})$,  
so using this for $B=\Ec_{<t}$ and $A$ the events $\{X^v_t(e)=1\},\{\hat{X}_t(e)=1\}$, the triangle inequality gives 
\begin{align*}
\Pr\big(X_t(e)\neq \hat{X}_t(e)\mid \Ec_{< t}\big)
\leq 4\Pr\big(\overline{ \Ec_{<t}}\big)+| \Pr(X^v_t(e)=1)-\Pr(\hat{X}_t(e)=1)|
\end{align*}
Note that $\Pr(\hat{X}_t(e)=1)=\pi^{\ord}_G(e\mapsto 1)$, so another application of triangle inequality gives 
\begin{equation}\label{eq:rvffv2}
\begin{aligned}
\Pr\big(X_t(e)\neq \hat{X}_t(e)\mid \Ec_{< t}\big)\leq4\Pr\big(\overline{ \Ec_{<t}}\big)+\big|\Pr(X^v_t(e)&=1)-\pi_{\B^+_r(v)}(e\mapsto 1)\big|\\
&+\big|\pi_{\B^+_r(v)}(e\mapsto 1)-\pi^{\ord}_G(e\mapsto 1)\big|.
\end{aligned}
\end{equation}
Since $G$ has WSM within the ordered phase at radius $r$, we have that
\begin{equation}\label{eq:rvffv1}
\big|\pi_{\B^+_r(v)}(e\mapsto 1)-\pi^{\ord}_G(e\mapsto 1)\big|\leq 1/(100 m).
\end{equation}
Moreover, let $T_v$ be the mixing time of the wired RC dynamics on $G[B_r(v)]$ and let  $N_v=|E(B_r(v))|\leq \Delta^{r+1}$. Since $r\leq \tfrac{1}{3}\log_{\Delta-1} n$, $G[B_r(v)]$ is $K$-treelike, so from Lemma~\ref{lem:mixtreelike}, with  $\hat{C}=O(1)$ denoting the constant there (and absorbing a couple of factors of~$\Delta$ into it), 
\[T_v\leq \hat{C} (N_v)^3 (q^4 \emm^\beta)^{\Delta r}\leq \hat{C} N_v\Delta^{2r}q^{4\Delta r}\emm^{\beta \Delta r}= \hat{C} N_v \Big(\Delta^{2M/\beta} q^{4\Delta M/\beta}\emm^{M \Delta}\Big)^{\log n}\leq \hat{C} N_v W^{\log n},\]
where in the last inequality we used that $\beta>\beta_0$, $\beta_0> \tfrac{1}{\Delta}\log q$ and $W=\Delta^{2M/\beta_0} \emm^{M \Delta(\Delta+1)}$. For $T=\Theta(n^{2+\log W})$,  we have
$T\geq 40T_v \tfrac{m}
{N_v}\log m$, so by Chernoff bounds, with probability $1-\exp(-n^{\Omega(1)})$, we have at least $10T_v \log m$ updates inside  $B_r(v)$ among  $t=1,\hdots, T$. For integer $k\geq 1$ the distance from stationarity after $kT_v$ steps is at most $(1/4)^k$, we obtain 
\begin{equation}\label{eq:rvffv}
\big|\Pr(X^v_T(e)=1)-\pi_{\B^+_r(v)}(e\mapsto 1)\big|\leq \exp(-n^{\Omega(1)})+ \emm^{-4\log m}\leq 1/m^3.
\end{equation}
Plugging \eqref{eq:rvffv1} and \eqref{eq:rvffv} into \eqref{eq:rvffv2} for $t=T$, and then back into \eqref{eq:rvffv3}, we obtain using $\Pr(\overline{\Ec_{<T}})\leq T\emm^{-\Omega(n)}$ that $\Pr(X_T\neq \hat{X}_T)\leq 5m T \emm^{-\Omega(n)}+m/m^3+1/100\leq 1/5$, as needed.
\end{proof}
\begin{remark}
The value $T$ in the proof of Theorem~\ref{thm:main1b} 
gives the running time $O(n^{C(\Delta)})$ in the statement.
In particular, $T = \Theta(n^{2 + \log W})$
where  $W=\Delta^{2M/\beta_0} \emm^{M \Delta(\Delta+1)}$,
$\beta_0 > \log(q)/\Delta$,
and $M = \Theta(\Delta \log \Delta)$ (see the start of Section~\ref{sec4}) so $C(\Delta) = O(\Delta^3 \log \Delta)$.
\end{remark}
To get the improved mixing time bounds of $O(n\log n)$ in Theorems~\ref{thm:main1c} and~\ref{thm:main1d} the reasoning is very similar. The main difference from the above proof is that for any vertex $v$ the mixing time $T_v$ is bounded by $T_v=O( N_v \log N_v)$, and therefore the above argument yields a mixing time upper bound of $O(n(\log n)^2)$ so a bit more care is needed to remove the extra $\log n$ factor using a log-Sobolev inequality, the details can be found in Section~\ref{sec:comp}.

Finally, note that Theorems~\ref{thm:main1b},~\ref{thm:main1c} and~\ref{thm:main1d} all apply to the critical case $\beta=\beta_c$ as well. The only difference at criticality is that the two phases coexist \cite{RCM-Helmuth2020} (i.e., each appears with $\Omega(1)$ probability), so in order to obtain a sample from $\pi_G$, one should output a sample for $\pi^{\ord}_G$ with some probability $Q$ and otherwise a sample from $\pi^{\dis}_G$. The value of $Q$ can be computed in time $\tilde{O}(n^2)$ by approximating the corresponding partition functions, by using, e.g., the algorithms in \cite{RCM-Helmuth2020,fastpolymers} (or even the RC dynamics itself). Precise results characterising the distribution of $Q$ can further be found in \cite[Theorems 2 \& 3]{RCM-Helmuth2020}; it is shown for example that $Q$ converges to $1/(q+1)$ as $q$ grows large.
\newcommand{\statecorcrit}{Let $\Delta\geq 5$ be an integer. 
Let $q$ be sufficiently large and $\beta = \beta_c$. 
W.h.p. over $G\sim \G_{n,\Delta}$
there is a polynomial-time randomised algorithm   for sampling from the random cluster model with parameters~$q$ and~$\beta$.}

\subsection{Organisation of the rest of the paper}\label{sec:dwefwe}
We first give a summary of how to obtain our WSM results for the ordered phase in the next section (Section~\ref{sec:wsmsummary}); this is the most involved part of our arguments. Then, in Section~\ref{sec:graph}, we revisit the polymer framework for the random-cluster model and show some technical results that we will need to carry out the WSM proofs in Sections~\ref{sec4} and~\ref{sec:WSMdis} (for the ordered and disordered regimes, respectively). We conclude with Section~\ref{sec:thmmain2} which has some left-over technical pieces needed for the proof of Theorem~\ref{thm:main1}. Appendix~\ref{app:mixtreelike} has, for completeness, a proof of the mixing-time bound on the tree stated in Lemma~\ref{lem:mixtreelike} (uses relatively standard arguments) and Appendix~\ref{app:boundary} reviews some standard monotonicity properties of the random-cluster model.  

\section{Proof outline of the WSM within the ordered phase}\label{sec:wsmsummary}

\subsection{Locally tree-like expanders}
Analogously to \cite{RCM-Helmuth2020}, we work a bit more generally with $\Delta$-regular expanders, which are also tree-like. The \textit{expansion profile} of an $n$-vertex graph $G=(V,E)$  for  $\epsilon>0$ is given by
\[\phi_G(\epsilon) \coloneqq \min_{S\subseteq V;\ 0 < |S| \leq \epsilon n} \frac{|E(S,V\backslash S)|}{\Delta|S|}.\]
Then the classes $G_{\Delta,\delta}$ and  $\G_{\Delta,\delta,K}$ are as follows.

\begin{definition}
Let $\Delta \geq 5$ be an integer, and $\delta\in(0,1/2), K>0$ be reals. 
$G_{\Delta, \delta}$ is the class of $\Delta$-regular graphs such that
$\phi_G(1/2) \geq 1/10$ and $\phi_G(\delta) \geq 5/9$.
$\G_{\Delta,\delta,K}$ is the class of 
all locally $K$-treelike 
graphs $G\in \G_{\Delta,\delta}$.
\end{definition}

We use the following lemma.

\begin{lemma}[{\cite[Proposition 37]{RCM-Helmuth2020}}]\label{lem:class}
Fix $\Delta \geq 5$. There is a constant $\delta\in(0,1/2)$  such that w.h.p. a uniformly random $\Delta$-regular graph belongs to $\G_{\Delta,\delta}$.
\end{lemma}

Lemma~\ref{lem:class} and Lemma~\ref{lem:treelike} show that there is also a positive integer~$K$ such that, w.h.p, $G\in \mathcal{G}_{\Delta,\delta,K}$.
Next we state an important property of expanders from \cite{trevisan2016expanders}. 
 
\begin{lemma}[{\cite[Lemma 2.3]{trevisan2016expanders}}]\label{lemm:big-component}
Let $G=(V,E)$ be a regular graph and consider $E'\subseteq E$ with $|E'|\leq\theta |E|$ for some $\theta\in(0,\phi_G(1/2))$. Then $(V,E\backslash E')$ has a component of size at least $\big(1 - \tfrac{\theta}{2 \phi_G(1/2)}\big) |V|$.
\end{lemma}

We use Lemma~\ref{lemm:big-component} to establish the existence of a giant component.

\begin{definition}The size of a component of a graph is the number of vertices in the component.
A \emph{giant component} in an $n$-vertex graph is a component whose size is greater than $n/2$.   Given a graph $G=(V,E)$ and a subset $F\subseteq E$,  $G[F]$ denotes the graph $(V,F)$.\end{definition}

\begin{definition}\label{def:eta}
Fix $\Delta \geq 5$. Fix $\delta \in (0,1/2)$ 
satisfying Lemma~\ref{lem:class}. Let  $\eta = \min(\delta/5, 1/100)$.
\end{definition}

\begin{corollary}\label{cor:big}
Fix integers $\Delta\geq 5$ and $K\geq 0$ and a real number $\delta\in(0,1/2)$. Let $G$ be a graph in $ \G_{\Delta,\delta,K}$ and let $\F$ be a configuration in $\Omega^\ord$ or a partial configuration with $|\In(\F)|\geq(1-\eta)|E|$. Then there is a giant component in $G[\In(\F)]$ whose size is at least $ (1-\delta)|V|$.
\end{corollary}
\begin{proof}
    Apply Lemma~\ref{lemm:big-component} with $E' = \Out(\F)$ and $\theta = \eta = \min(\delta/5,1/100)$. Note $|\Out(\F)|\leq \eta |E|$ and $\phi_G(1/2)\geq 1/10$. Thus the lemma say that $G[\In(\F)]$ has a component of size at least $\left(1-\frac{\delta / 5}{2\cdot 1/10}\right)|V| = (1 - \delta)|V| > |V|/2$.
\end{proof}

\subsection{Sketch of proof of Theorem~\ref{thm:ordered}} 

Let $\Delta\geq 5$ be an integer.
Consider any sufficiently large $q$ and any
$\beta \geq \beta_c$. 
For sufficiently large~$n$, choose a ``radius'' $r \approx \tfrac{1}{\beta} \log n$ and let
$G=(V,E)\sim \G_{n,\Delta}$. 
Fix a vertex $v\in V$ and an edge $e$ incident to~$v$.
We wish to show, with sufficiently high probability, that $\Vert{\pi_{\B_r^+(v)}(e \mapsto \cdot ) - \pi^{\ord}_G(e\mapsto \cdot) }\Vert_{\tv} \leq  {1}/{(100 |E|)}$. 

Our goal is essentially to
construct a coupling of $\F^+\sim\pi_{\B^+_{r}(v)}$ and $\F^\ord\sim\pi^\ord$, such that  $\Pr(\F^+(e)\neq\F^\ord(e))  $ is sufficiently small. In order to construct the coupling, we take advantage of the fact that $G[B_r(v)]$ is locally tree-like. In fact, we identify a suitable subgraph of~$G[B_r(v)]$ without cycles and restrict the coupling to this subgraph.

Consider a breadth-first search from~$v$ in~$G[B_r(v)]$.
Let $T_0$ be the rooted tree consisting of all forward edges in this breadth-first search.
All other edges in~$B_r(v)$ are called ``excess edges''. 
W.h.p., since $G\sim \G_{n,\Delta}$, there are at most $K$ excess edges in $B_r(v)$ for some absolute constant $K>0$. In particular,
since $G$ is locally tree-like,
we can identify integers   $r_1$ and $r_2$ satisfying $r \geq r_1 > r_2 \geq 0$ such that 
$E(B_{r_1}(v))\setminus E(B_{r_2}(v))$ contains no excess edges and $r_1 - r_2  \geq r/(2K)=\Omega(r)$.
The fact that $r_1-r_2 = \Omega(r)$ ensures that 
$B_{r_1}(v) \setminus B_{r_2}(v)$ is a sufficiently large subgraph of~$G$, and the coupling focuses on this subgraph.

In order to describe the coupling process we need a small amount of notation.
A \emph{partial configuration} $\F$ is a map from the edges of~$G$ to the set $\{0,1,*\}$.
In-edges and out-edges (those that are mapped to~$1$ or to~$0$) are   ``revealed'' and edges that are mapped to~$*$ are ``unrevealed''. 
A refinement of a partial configuration is obtained by revealing more edges.
We use $\F \subseteq \F'$ to denote the fact that $\F'$ refines~$\F$.

In the coupling, we generate a sequence of edge subsets $F_0 \subseteq F_1 \subseteq \cdots \subseteq E$
such that, after iteration~$i$, the edges in~$F_i$ are revealed.
We also construct two sequences of partial configurations 
$\F_0^+\subseteq\F_1^+\subseteq\dots\subseteq\F^+$ and 
$\F_0^\ord\subseteq\F_1^\ord\subseteq\dots\subseteq\F^\ord$, maintaining the invariant that the revealed edges in $\F_i^\ord$ and $\F_i^+$ are exactly the edges in~$F_i$.
The coupling will have the crucial property that $\F^\ord \sim \pi^\ord$ and $\F^+ \sim \pi_{\B_{r_1}^+(v)}$

\begin{itemize}
\item The process starts with iteration $i=0$. The initial set $F_0$ of revealed edges is all edges except those in $E(B_{r_1}(v))$.
In $\F_0^\ord$ these revealed edges are sampled from the projection $\pi^\ord_{F_0}$ of~$\pi^\ord$ onto~$F_0$.
It is likely that the configuration $\F_0^\ord$ has at least $(1-\eta)|E|$ in-edges. If not,
then the coupling terminates (unsuccessfully), generating $\F^\ord$ and $\F^+$ from the right distributions. We will show that the probability of this unsuccessful termination is low.
On the other hand, if $\F_0^\ord$ has a least $(1-\eta)|E|$ in-edges, then we are off to a good start.
All configurations refining $\F_0^\ord$ are in $\Omega^\ord$, so the projection of~$\pi$ and~$\pi^\ord$ onto subsequent edges that get revealed are the same (making it easier to continue the coupling).
At this point $\F_0^+$ is taken to be the configuration with revealed edges~$F_0$ where all revealed edges are in-edges.

\item After iteration $i=0$, iterations continue with $i=1,2,\ldots$ until an edge is revealed whose distance from~$v$ is at most~$r_2$
or until the in-edges in $\F_i^\ord$ induce a giant component, and this giant component contains all vertices on the boundary of~$F_i$. We will show that it is very unlikely that an edge at distance at most~$r_2$ from~$v$ is reached. So it is likely the giant component in~$\F_i^\ord$ contains all vertices on the boundary of~$F_i$.
This is a good situation because the conditional distribution of~$\pi$, conditioned on 
refining~$\F_i^\ord$ and the conditional distribution of~$\F^+$, conditioned on
refining~$\F_i^+$ induce the same distribution on edges incident to~$v$, which enables us to show that
 $\Pr(\F^+(e)\neq\F^\ord(e))  $ is sufficiently small.

\item The process at  iteration~$i+1$ is as follows.
$W_i$ is taken to be the set of all vertices on the boundary of~$F_i$ whose components 
(induced by the in-edges in $\F_i^\ord$) are all small.
By ``boundary'' we mean that vertices in $W_i$ are adjacent to revealed edges, and to unrevealed edges.
If $W_i$ is empty, then the coupling finishes. Otherwise, a vertex $w_i\in W_i$ is chosen to be as far from~$v$ as possible. The edges in the subtree of~$T_0$ below the parent of~$w_i$ are revealed in~$F_{i+1}$.

\end{itemize}

The main remaining ingredient in the proof is showing that the unsuccessful terminations of the coupling are unlikely.  To do this, we use the polymer framework of~\cite{RCM-Helmuth2020}. 
(Ordered) polymers are defined using an inductive definition.
For a  set of edges $A\subseteq E$, let $\B_0(A) = A$, and inductively for $j = 0,1,2,\dots$ define $\B_{j+1}(A)$ to be the set of all edges such that they are either in $\B_j(A)$ or edges that are incident to a vertex that has at least $ {5\Delta}/{9}$ incident edges in $\B_j(A)$. Let $\B_\infty(A) = \bigcup_{j\in\N} \B_j(A)$. 
An ordered polymer of a configuration~$\F$ is a connected component of $B_\infty(\Out(\F))$.
The bulk of the work is to prove the following lemma, which is repeated in Section~\ref{sec:coupling} (with more detail) as Lemma~\ref{lemm:agree-on-edges-of-v-or-large-polymer}.

\begin{lemma}\label{introlem} 
Fix $\Delta\geq 5$ and $K, M>0$.  Suppose that $\beta \geq 3 M$.
Suppose that $n$ is sufficiently large so  that $r:= \tfrac\M\beta \log_{\Delta-1} n >K$
and $|B_{r}(v)| \leq 9 \Delta n/200$. 
Define $r_1$ as above.
Let $\F^\ord$ and $\F^+$ be generated by the process. Then at least one of the following conditions holds.
\begin{enumerate}
\item   $\F^\ord$ and $\F^+$ agree on the edges that are incident to~$v$.
\item   $|\In(\F^\ord)\setminus E(B_{r_1}(v))| < (1-\eta) |E|$.
\item   $\F^\ord$ contains a polymer of size at least $\frac{r}{400\Delta(1+K)}-1$.
\end{enumerate} 
\end{lemma} 

To complete the proof of Theorem~\ref{thm:ordered}, we show that items~2 and~3 are unlikely.
The proof that item~2 is unlikely, Lemma~\ref{cor:eee}, follows from the slack specified in Equation~\eqref{eq:margin} of Lemma~\ref{lem:dominant}.
The proof that item~3 is unlikely, Lemma~\ref{cor:bbb}, follows from an analysis on the size of polymers by adapting appropriately the cluster expansion techniques of~\cite{RCM-Helmuth2020} (a bit of extra work is needed there to capture the $1/\beta$ dependence in the size of the polymer, see Lemma~\ref{cor:bbb}).

The proof of WSM for the disordered phase (Theorem~\ref{thm:RRwsmB}) follows a similar strategy but the details are substantially simpler (due to a more straightforward notion of polymers), the argument can be found in Section~\ref{sec:WSMdis}.

\section{Polymers for RC on expander graphs}\label{sec:graph}
All logarithms in this paper are with base~$\emm$. Given a graph $G=(V,E)$, a subset $U\subseteq V$ and a subset $F\subseteq E$,
we use $G[U]$ to denote the subgraph of~$G$ induced by~$U$ and we use $G[F]$ to denote the graph $(V,F)$.

\subsection{The RC model on expanders in the ordered regime}\label{sec:RCstuff}
We use the following Lemma from \cite{RCM-Helmuth2020}.

\begin{lemma}[{\cite[Lemma 8]{RCM-Helmuth2020}}]\label{lemm:rcm-components+edges-bound}\label{lem8}
Let $\Delta\geq 5$,  $\delta\in(0, 1/2)$, and $\zeta\in(0, 1/2)$ be reals. Suppose that $G=(V,E)\in\G_{\Delta,\delta}$ satisfies $|V|\geq 360/(\zeta \delta)$ and $A\subset E$ is such that $\zeta|E| \leq |A| \leq (1-\zeta)|E|$. Then
${c(A)}/{|V|} + {|A|}/{|E|} \leq 1 - \zeta/40$
where $c(A)$ is the number of connected components in $G[A]$.
\end{lemma}
\begin{remark}
\cite[Lemma 8]{RCM-Helmuth2020} is proved for the special case of $\zeta = \min(1/100,\delta/5)$, however the proof works for any positive $\zeta<1/2$.
\end{remark}

\begin{lemma}[{\cite[Theorems 1 and 2]{RCM-Helmuth2020}}]\label{lem:slack} 
Let $\Delta \geq 5$ be an integer. Fix $\delta \in (0,1/2)$ satisfying Lemma~\ref{lem:class}. Let $\eta = \min \{\delta/5,1/100\}$.
Let $q_0  \geq e^{21 \Delta/\eta}$ be sufficiently large that 
Theorems~1 and~2 of~\cite{RCM-Helmuth2020} hold. There is a positive number $\zeta < \eta$ such that 
w.h.p. for $G\sim \G_{n,\Delta}$, it holds that 
\begin{equation*} 
\begin{aligned}
\mbox{for $\beta\geq \beta_c$}, &\quad  \pi_G^\ord\Big(|\In(\F)|\leq (1-\zeta)|E|\Big)=\emm^{-\Omega(n)},\\
\mbox{for $\beta\leq \beta_c$}, &\quad  \pi_G^\dis\Big(|\In(\F)|\geq \zeta|E|\Big)=\emm^{-\Omega(n)}.
\end{aligned}
\end{equation*} 
\end{lemma}
\begin{proof}
We prove the first bound, the second is similar.
Choose $\zeta = 20 \Delta/\log(q_0)$. Note that $\zeta < \eta$, as desired.
Let  $\Omega' = \{\F \in \Omega^\ord: |\In(\F)|\leq (1-\zeta)|E|\}$ and let 
$$Z' = \sum_{\F \in \Omega'} w_G(\F) = \sum_{\F\in \Omega'} q^{c(\F)} (\emm^\beta - 1)^{|\In(\F)|}.$$
Let $z = \max\{q,(\emm^{\beta}-1)^{\Delta/2}\}$. Since $Z \geq z^n$ every $\F\in \Omega$ has
$$w_G(\F)/Z \leq z^{c(\F) + 2|\In(\F)|/\Delta  - n} = z^{n(c(\F)/n + |\In(\F)|/|E|  - 1)}
.$$
Since $\zeta < \eta < 1/2$,  for any
$\F\in \Omega'$, 
$\zeta |E| \leq (1-\eta)|E| \leq |\In(\F)| \leq (1-\zeta)|E|$.
By Lemma~\ref{lem:class},   w.h.p., $G\in \mathcal{G}_{\Delta,\delta}$.
By Lemma~\ref{lem8}, for $n\geq 360/(\zeta \delta)$, we have
$c(\F)/n + |\In(\F)|/|E| \leq 1-\zeta/40$ so
$Z'/Z \leq 2^{n\Delta/2} z^{n(-\zeta/40)}$, which gives 
$Z'/Z = \emm^{-\Omega(n)}$  since
$q > \emm^{20 \Delta/\zeta}$. 
To finish, we will show that $Z/Z^\ord = O(1)$, so
that $\frac{Z'}{Z^\ord} = \frac{Z'}{Z} \times \frac{Z}{Z^\ord} = \emm^{-\Omega(n)} O(1) = \emm^{-\Omega(n)}$.
If $\beta > \beta_c$, item~3 of \cite[Theorem 1]{RCM-Helmuth2020}  says that there are positive numbers~$n_0$ and~$\xi$ such that, for $n\geq n_0$,
$(1/n) \log ((Z-Z^\ord)/Z) \leq -\xi$.
So for $n\geq n_0$,  $Z^\ord/Z \geq 1 - e^{-\xi n}$, which is at least~$1/2$ for $n\geq 1/\xi$.
If $\beta = \beta_c$,  
then items~1 and~3 of~\cite[Theorem 2]{RCM-Helmuth2020} show that 
$Z^\ord/Z$ converges in distribution to a random variable $Q/(Q+1)$ where (say) $Q\geq q/2$ for $q\geq q_0$. So for sufficiently large~$q$, $Z^\ord/Z = O(1)$.
 \end{proof}

\subsection{Disordered polymers}

We will use \textit{disordered polymers} from \cite[Section~2.3]{RCM-Helmuth2020}. 
Given a graph $G = (V,E)$ and a configuration~$\F\in \Omega^\dis$, a \textit{disordered polymer} of~$\F$ is a connected component of $G[\In(\F)]$.
Since $\F \in \Omega^\dis$, each such connected component contains at most $\eta|E|$ edges.

Each disordered polymer $\gamma=(V(\gamma),E(\gamma))$ has an associated weight given by $ w_\gamma^\dis := q^{1-|V(\gamma)|} (e^{\beta}-1)^{|E(\gamma)|}
$.  We observe the following relation between the weight of a disordered configuration and the weights of its polymers.
\begin{observation}\label{obs:obs}
For $\F\in\Omega^\ord$, let $\Gamma^\dis(\F)$ be the set of disordered polymers of $\F$. Then
$$w_G(\F) = q^n \prod_{\gamma\in\Gamma^\dis(\F)} w_\gamma^\dis.$$
\end{observation}
This observation follows from the fact each vertex of $G$ is contained in some polymer of $\F$, each in-edge of $\F$ is an edge of some polymer of $\F$, and each polymer of $\F$ corresponds to one component of $G[\In(\F)]$. 

Based on the work by Helmuth, Jenssen and Perkins \cite[Proposition 11]{RCM-Helmuth2020}, we can upper bound the probability that $\F\sim\pi^\dis$ contains a large polymer.

\begin{lemma}\label{lemm:polymer-weight-disordered}
Let $\Delta \geq 5$ and $K\geq 0$ be integers, $\delta\in(0, 1/2)$ be real. Let $C$ be a constant. Then, for all sufficiently large $q$, the following holds for all $\beta \leq \log(q^{2.1/\Delta}+1)$. Let $n$ be sufficiently large. For $G\in\G_{\Delta,\delta,K}$ with $n=|V(G)|$, for $s:=C \log n$,
$$\pi^{\dis}_G\big(\mbox{$\F$ contains a polymer with at least $s$ edges}\big)\leq n^{-3/2}/2.$$
\end{lemma}
\begin{proof}
Let $q$ and $n$ be large enough so that Proposition 11 from \cite{RCM-Helmuth2020} applies, and so in particular \cite[Equation (13)]{RCM-Helmuth2020} holds, and also that $s > 1$ and $q\geq e^{4\Delta(\frac{5}{2C}-1)}$.

For a disordered configuration $\F$, let $\Gamma^\dis(\F)$ be the set of disordered polymers of $\F$. 
 
First we will bound, for a particular polymer $\gamma$ with $|E(\gamma)| \geq 1$, $\Pr(\gamma\in\Gamma^\dis(.))$. Suppose $\F\in\Omega^\dis$ has $\gamma\in\Gamma^\dis(\F)$. By Observation~\ref{obs:obs}, 
$w_G(\F) = q^n \prod_{\gamma\in\Gamma^\dis(\F)} w_\gamma^\dis$.
Let $\F'$ be the (disordered) configuration defined by $\In(\F') = \In(\F)\backslash E(\gamma)$. Since $E(\gamma)\neq\emptyset$, $\F\neq\F'$. By construction,
$\F'\in\Omega^\dis$. Since vertices in~$\gamma$ are not incident to edges in $\In(\F')$, $\Gamma^\dis(\F') = (\Gamma^\dis(\F)\setminus\{\gamma\})\cup_{v\in\gamma} \{v\}$.
A polymer~$\gamma'$ consisting of a single vertex has weight $w_{\gamma'}^\dis = 1$.
Thus, $w_G(\F') = w_G(\F)/w_\gamma^\dis$.
Let $\Omega_\gamma$ be the set of all disordered configurations containing $\gamma$ as a polymer.
Then 
$Z^\dis \geq \sum_{\F\in\Omega_\gamma} (w_G(\F) + w_G(\F')) = (1 + 1/w_\gamma^\dis)\sum_{\F\in\Omega_\gamma} w_G(\F)$.
Thus, the probability that $\F\sim\pi^\dis$ contains the polymer~$\gamma$ is at most $\frac{1}{1 + 1/w_\gamma^\dis} \leq w_\gamma^\dis$.

 Next,  for $v\in V(G)$ 
 let $p_v$ be the probability that $\F\sim \pi_G^\dis$ has a polymer with at least $s$ edges containing the vertex~$v$ vertex~$v$.  By a union bound, $p_v \leq P_v := \sum_{\gamma\ni v:  |E(\gamma)|\geq s} w_\gamma^\dis$.
 Let $\rho = \log(q)/(4 \Delta)$.
By   \cite[Equation (13)]{RCM-Helmuth2020},
$$ \sum_{\gamma\ni v: |\gamma|>1} e^{(1+\rho)|E(\gamma)|} w_\gamma^\dis \leq \tfrac{1}{2}.$$
Thus also
$$e^{(1+\rho)s} P_v = 
e^{(1+\rho)s}\sum_{\gamma\ni v: |E(\gamma)|\geq s} w_\gamma^\dis \leq \sum_{\gamma\ni : |E(\gamma)|\geq s} e^{(1+ \rho)|E(\gamma)|} w_\gamma^\dis \leq \tfrac 12.$$
Hence $P_v \leq (1/2) e^{-(1+\rho)s}$. By a union bound over all vertices, the probability in the statement of the lemma is at most $\sum_v P_v \leq (n/2) 2^{-(1+\rho)s}= 
(1/2) n^{1-C(1+\rho)}
$.  The lemma follows by requiring $q$ to be sufficiently large with respect to~$\Delta$ and~$C$ that $C(1+\rho) -1 \geq 3/2$.
 
\end{proof}

\subsection{Ordered polymers}

 We use the definition of \textit{ordered polymers} from \cite[Section~2.4.2]{RCM-Helmuth2020}. Let $G=(V,E)$ be a graph of max degree $\Delta$ with $n=|V(G)|$ and $m=|E(G)|$. For a  set of edges $A\subseteq E$, Let $\B_0(A) = A$, and inductively for $i = 0,1,2,\dots$ define $\B_{i+1}(A)$ to be the set of all edges such that they are either in $\B_i(A)$ or edges that are incident to a vertex that has at least $ {5\Delta}/{9}$ incident edges in $\B_i(A)$. Let $\B_\infty(A) = \bigcup_{i\in\N} \B_i(A)$. It is shown in \cite[Lemma 12]{RCM-Helmuth2020} that $|\B_\infty(A)|\leq 10|\B_0(A)|$.

 An ordered polymer is a connected subgraph $\gamma=(V(\gamma),E(\gamma))$ of $G$ together with an edge configuration $\ell: E(\gamma)\rightarrow \{0,1\}$ subject to: (i) $\Out(\ell)\leq \eta m$, and (ii) $\B_{\infty}(\Out(\ell))=E(\gamma)$. We let $E_{u}(\gamma)$ be the set of unoccupied edges $\Out(\ell)$ of the polymer and $c'(\gamma)$ be the number of components in the graph $G[E\backslash E_{u}(\gamma)]$ with fewer than $n/2$ vertices. The size of $\gamma$ is defined as $|E_\gamma|$, whereas the weight of $\gamma$ is defined as $w_\gamma^\ord := q^{c'(\gamma)} (e^\beta-1)^{-E_u(\gamma)}$.

 Given a (partial) configuration $\F$, the set of ordered polymers of $\F$, denoted by $\Gamma(\F)$, consists of the connected components of $G[\B_\infty(\Out(\F))]$, each with the labelling on the edges induced by the corresponding assignment in $\F$.

We will use a couple of polymer properties from \cite{RCM-Helmuth2020}. First, we note the following connection between the weight of a configuration $\F$ to the weight of its polymers $\Gamma(\F)$.

\begin{lemma}\label{lemm:polymer-weight-configuration-weight}  
Fix $\Delta\geq 5$, $K\geq 0$ and $\delta\in(0, 1/2)$. Let $G\in\G_{\Delta,\delta,K}$ and $\F\in\Omega^\ord$.  Then
$w_G(\F) = q(\emm^\beta - 1)^{|E|} \prod_{\gamma\in\Gamma(\F)} w_\gamma^\ord$.
\end{lemma}
\begin{proof}
   Note that $\Out(\F) = \bigcup_{\gamma\in\Gamma(\F)} E_u(\gamma)$, so $(\emm^\beta-1)^{|E| - \sum_{\gamma\in\Gamma(\F)}|E_u(\gamma)|} = (\emm^\beta-1)^{|\In(\F)|}.$
    Hence, by \cite[Lemma 21]{RCM-Helmuth2020}, $q^{\sum_{\gamma\in\Gamma(\F)} c'(\gamma)} = q^{c(G[\In(\F)]) - 1}$, so the result follows.
\end{proof}

\begin{lemma}\label{lemm:large-q-large-polymer-unlikely}\label{cor:bbb}
Let $\Delta \geq 5$ and $K\geq 0$ be integers, $\delta\in(0, 1/2)$ be real. Then, for all sufficiently large $q$, the following holds for all $\beta \geq \log(q^{1.9/\Delta}+1)$. For $G\in\G_{\Delta,\delta,K}$ with $n=|V(G)|$, for $s:={2000}\log(n)/\beta$,
\[
    \pi^{\ord}_G\big(
    \mbox{$\F$ contains a polymer with size at least $s$} \big)\leq 2n^{-3/2}.
\]
\end{lemma}
\begin{proof}
This proof closely follows the proof of Proposition 15 from \cite{RCM-Helmuth2020}.
Suppose that $q$ is large enough so that $q^{1.9/\Delta}\geq (2\emm\Delta)^{400}$ and fix arbitrary $\beta \geq \log(q^{1.9/\Delta}+1)$.

Consider an arbitrary ordered polymer~$\gamma$. For any configuration $\F\in \Omega^{\ord}$ with $\gamma\in \Gamma(\F)$, by
Lemma~\ref{lemm:polymer-weight-configuration-weight}, it holds that $w_G(\F) \leq q(\emm^\beta - 1)^{|E|} w_\gamma^\ord$.
Note that for any $\F\in\Omega^\ord$ such that $\gamma\in\Gamma(\F)$, there is a configuration $\F'$ with $\Gamma(\F') = \Gamma(\F)\setminus\{\gamma\}$ (obtained by setting $\In(\F') = \In(\F)\cup E_u(\gamma)$), which therefore has weight $w_G(\F') = w_G(\F)/w_\gamma^\ord$.

Let $\Omega_\gamma$ be the set of all ordered configurations with a polymer $\gamma$. We can lower bound $Z^\ord$ by $\sum_{\F\in\Omega_\gamma} (w_G(\F) + w_G(\F')) = (1 + \frac{1}{w_\gamma^\ord})
\sum_{\F\in\Omega_\gamma} w_G(\F)$. Therefore, 
$$
\pi^\ord(\gamma\in\Gamma(\cdot)) = \frac{\sum_{\F\in\Omega_\gamma} w_G(\F)}{Z^\ord} \leq \frac{1}{1 + \frac{1}{w_\gamma^\ord}}\leq w_\gamma^\ord,
$$
so we conclude that the probability that any fixed polymer $\gamma$ is a polymer of $\F\sim\pi^\ord_G$ is at most $w_\gamma^\ord$.

Let $u$ be a vertex in~$V(G)$ and denote by $\Gamma_u$ the set of polymers containing $u$. Let $p_u$ be the probability that $u$ is contained in a polymer of size at least~$s$ for $\F\sim\pi^\ord_G$. By a union bound over the polymers in $\Gamma_u$, 
We have that $p_u\leq P_u$, where 
\[
P_u:= \sum_{\gamma\in \Gamma_u;\  |E(\gamma)|\geq s} w_\gamma^\ord 
\]

By \cite[Lemma 14]{RCM-Helmuth2020}, 
$c'(\gamma)\leq 9|E_u(\gamma)|/(5 \Delta)$. The bound on~$\beta$ ensures that 
\[w_\gamma^\ord = q^{c'(\gamma)} (\emm^\beta-1)^{-|E_u(\gamma)|} \leq (\emm^\beta-1)^{(\frac{9}{5\Delta}\cdot \frac{\Delta}{1.9} - 1)|E_u(\gamma)|}\leq (\emm^\beta-1)^{-|E_u(\gamma)|/20}.\]

Hence we get
\[
P_u \leq \sum_{\gamma\in \Gamma_u;\  |E(\gamma)|\geq s}  (\emm^\beta-1)^{-|E_u(\gamma)|/20}.
\]
By \cite[Lemma 12]{RCM-Helmuth2020}, for any ordered polymer $\gamma$, $|E(\gamma)|\leq 10|E_u(\gamma)|$. Thus  
\[
    P_u \leq \sum_{k\geq s/10}\  \ \sum_{\gamma\in \Gamma_u;\ |E_u(\gamma)| = k} (\emm^\beta-1)^{-|E_u(\gamma)|/20}.
\]
By \cite[Proof of Proposition 15]{RCM-Helmuth2020}, the number of polymers with $k$ unoccupied edges containing a particular vertex is at most $(2\emm\Delta)^{10k}$, hence
\[ P_u\leq \sum_{k\geq s/10}\big((2\emm\Delta)^{10}(\emm^{\beta}-1)^{-1/20}\big)^{k}.\]
By the lower bounds on $\beta$ and $q$, we have $\emm^{\beta}-1\geq q^{1.9/\Delta}\geq (2\emm\Delta)^{400}$ so 
\[ P_u\leq \sum_{k\geq s/10} (\emm^{\beta}-1)^{-k/40} \leq  \sum_{k\geq s/10} \emm^{-\beta k/80}.\]
The last sum is a geometric series with ratio $\emm^{-\beta/80}<1/2$, so it is upper bounded by twice its largest term. Since $s=\tfrac{2000}{\beta}\log n$, it follows that $P_u\leq 2/n^{5/2}$.
By a union bound over the vertices of $G$, the probability that there is a polymer of $\F\sim \pi^{\ord}_G$ which has size at least~$s$ is  therefore  at most $2/n^{3/2}$, as required.
\end{proof}

Furthermore, we can prove the following lemma based on the proof of \cite[Lemma 21]{RCM-Helmuth2020}.

\begin{lemma}\label{lemm:vertices-in-small-components-have-all-edges-in-single-polymer}\label{lem:shortname}
Fix $\Delta \geq 5$, $K\geq 0$ and $\delta \in (0,1/2)$. Suppose $n\geq 2/\delta$.
Let $G=(V,E)$ with $|V|=n$ be a graph in $\G_{\Delta,\delta,K}$. Consider $\F\in\Omega^\ord$. Let $\kappa$ be the giant component of $G[\In(\F)]$ (which exists by Corollary~\ref{cor:big}).
\begin{enumerate}
\item \label{item:one} 
If $u\notin \kappa$ and $e$ is incident to~$u$ then $e$ is in a polymer   of $\F$.
            
\item \label{item:two}
Let $S$ be a non-empty  set of edges with $|S| <  {9\Delta n}/{200}$.
Let $\kappa'$ be a component of $G[\In(\F)\backslash S]$ with
$|V(\kappa')| < n/2$. Then all but at most $45\Delta|S|$ vertices of $\kappa'$ are such that all their incident edges belong to a polymer of~$\F$. 
            
\item \label{item:three} 
For~$S$ and~$\kappa'$ as in Item~\ref{item:two},
there are at most $50\Delta^2|S|$ components in $G[\B_\infty(\Out(\F))]$ containing a vertex of $\kappa'$, i.e., there are  at most $50\Delta^2|S|$ polymers of~$\F$ containing vertices of $\kappa'$.
    \end{enumerate}
\end{lemma}
\begin{proof} 
We start with the proof of Item~\ref{item:one}.
Consider a non-giant component $\kappa'$ of $G[\In(\F)]$. Let $\tau$ be the set of vertices in $\kappa'$ which have an incident edge that is not in a polymer of~$\F$.  
For contradiction, assume that $\tau$ is non-empty.       
By Corollary~\ref{cor:big}, $|\kappa'|\leq \delta|V|$ and thus $|\tau|\leq\delta|V|$. 
The definition of~$\tau$ 
and the fact that out edges are in polymers imply
that every edge in the cut $(\tau, V\backslash \tau)$  is in some polymer. By the expansion properties $|E(\tau,V \backslash \tau)|\geq 5\Delta|\tau|/9$, hence  there is a vertex  $u\in \tau$ with at least $ {5\Delta}/{9}$ incident edges in the cut, and thus, from the definition of polymers, all edges adjacent to~$u$ are in polymers, contradicting that~$u$ is in $\tau$.
    
Before proving Item~\ref{item:two}, we use it to prove Item~\ref{item:three}. Let $\kappa'$ be a non-giant component in $G[\In(\F)\backslash S]$. As before, let $\tau$ be the set of vertices in $\kappa'$  which have an incident edge that is not in a polymer of~$\F$. By Item~\ref{item:two},   $|\tau| \leq 45\Delta |S|$. Let $E_\tau$ be the set of edges that are incident to vertices in~$\tau$. Note that $G[V(\kappa'),E(\kappa')]$ is connected, and removing an edge from a graph can increase the number of components by at most one. Thus, $(V(\kappa'),E(\kappa')\setminus E_\tau)$ has at most $1+\Delta |\tau|$ components. Any edges in these components  are in polymers of~$\F$. Thus there are at most $1 + \Delta|\tau| \leq 50\Delta^2|S|$ components of $G[\B_\infty(\Out(\F))]$ that contain vertices in $V(\kappa')$,  proving Item~\ref{item:three}.
 
To finish, we prove Item~\ref{item:two}.
First we bound the size of $\kappa'$. Since $\F \in \Omega^\ord$, $|\In(\F)\backslash S| \geq (1 - \eta) |E| - |S|$. 
Since $\eta \leq 1/100$, $\eta + \frac{|S|}{|E|} < 1/10 \leq \phi_G(1/2)$, hence we can apply Lemma \ref{lemm:big-component} with $E' = \Out(\F) \cup S$ and $\theta = |E'|/|E|$. We find that $(V,\In(\F) \backslash S)$ has a component of size at least $(1 - \frac{\theta}{2 \phi_G(1/2)}) n \geq (1 - {5 \theta} ) n$ so the number of vertices in $\kappa'$ is at most $5 \theta n \leq 5(\eta + |S|/|E|)n \leq 
\delta n + 10|S|/\Delta$. Thus, $|\tau| \leq \delta n + 10 |S|/\Delta$.
Next let 
$\rho = \lceil 10 |S|/\Delta\rceil$ and let
$U$ be a subset of~$\tau$ of size $|\tau| -  \rho \leq \delta n$.
Then
\begin{align*} 
|E(\tau, V\backslash \tau)|
&\geq |E(U,V\setminus \tau)| \\
&\geq |E(U,V\setminus U)| - |E(U,\tau\setminus U)|\\
&\geq \tfrac59 \Delta |U| -  \rho \Delta =
\tfrac59 \Delta |\tau| - 14 \rho \Delta /9.
\end{align*}

By the definition of~$\tau$, all edges in 
$E(\tau,V(\kappa')\setminus \tau)$ are in polymers of~$\F$.
Since $\kappa'$ is a component of $G[\In(\F)\setminus S]$,
there are at most $|S|$ edges in $\In(\F)$ between~$\tau$
and $V\setminus V(\kappa')$.
Thus the number of edges in $E(\tau, V\setminus \tau)$ that are in polymers is at least $|E(\tau, V\backslash \tau)| - |S|$. 
This is at least 
$\tfrac59 \Delta |\tau| - 14 \rho \Delta /9 - |S| \geq \tfrac59 \Delta |\tau| - 149 |S|/9 - 14\Delta/9 \geq \tfrac59 \Delta |\tau| - 45 \Delta|S|/9$,
where the last inequality follows from $\Delta\geq 5$, $|S|\geq 1$. 

Averaged over all vertices $w\in \tau$,
the number of edges in $E(\tau,V\setminus \tau)$ that are in polymers and
are incident to~$w$
is at least $\tfrac59 \Delta  - 45 \Delta|S|/(9\tau)$.
We may assume $\tau$ is non-empty, otherwise Item~\ref{item:two} directly follows. 
Thus, there is a vertex $w\in \tau$ 
such that the number of edges in $E(\tau,V\setminus \tau)$ that are in polymers and
are incident to~$w$
is at least $\tfrac59 \Delta  - 45 \Delta|S|/(9\tau)$.
Since this number is an integer, the number of edges in $E(\tau,V\setminus \tau)$ that are in polymers and
are incident to~$w$
is at least 
$$ W := \Big\lceil \frac59 \Delta  - \frac{45\Delta |S|}{9|\tau|}\Big\rceil. $$ 

Suppose for contradiction that $|\tau| > 45\Delta |S|$. Then the term that is subtracted in the definition of~$W$ is less than~$1/9$. On the other hand, if the first term, $5\Delta/9$ is not an integer, then its non-integer part is at least~$1/9$.
Either way, $W = \lceil 5 \Delta/9\rceil \geq 5\Delta/9$. As in the proof of Item~\ref{item:one}, all incident edges of $w$ hence must be a in a polymer, contradicting that $w\in\tau$.
\end{proof}

\section{Proving WSM within the ordered phase}\label{subsec:wsm-ordered}\label{sec4}

In this section we will prove Theorem~\ref{thm:RRwsm}. We will use the following notation.
\begin{definition}
Let $G=(V,E)$ be a graph. Given vertices $u,v\in V$, 
let $d_G(v,u)$ be the distance in~$G$ from~$v$ to~$u$. For $U\subseteq V$ and $F\subseteq E$, denote by $d_G(v,U)$ the min distance in~$G$ from~$v$ to any vertex in~$U$, and by $d_G(v,F)$ the min distance in~$G$ to any endpoint of any edge in~$F$.
Let $\partial_G(F)$ be the set of vertices in~$V$ incident to both an edge in~$F$ and an edge in $E\backslash F$. For a rooted tree $T$ and a vertex $u\in V(T)$, let $T_u$ denote the subtree of $T$ rooted at $u$.
\end{definition}
Fix positive integers $\Delta \geq 5$ and $K\geq 0$. 
Let $M = M(K,\Delta)=2000 \times 400 \times (2+K) \Delta\log(\Delta-1)$. Fix $q$ sufficiently large and $\beta \geq  \beta_c$.
Fix a graph $G=(V,E)\in\G_{\Delta,\delta,K}$ with $n=|V|$ sufficiently large and fix a vertex $v\in V$.  Let
$r := \min(\tfrac\M\beta, \tfrac 13) \log_{\Delta-1} n$.

\begin{definition}\label{def:excess-edges-and-tree}
Consider a breadth-first search from~$v$ in~$G[B_r(v)]$.
Let $T_0$ be the rooted tree consisting of all forward edges in this breadth-first search.
We refer to edges in $E(T_0)$ as the tree edges.
We refer to the edges in $\Ec := E(B_r(v)) \backslash E(T_0)$ as \emph{excess edges}.
\end{definition}

\begin{observation}\label{obs:excess}
For every edge $\{u,u'\}$ in~$\Ec$, $d_{T_0}(v,u)$ and $d_{T_0}(v,u')$ differ by at most one. For any  $u\in B_r(v)$, $d_G(v,u) = d_{T_0}(v,u)$. 
\end{observation}

Since $\G_{\Delta,\delta,K}$, $|\Ec|\leq K$ so the ball $B_r(v)$ is close to being a tree. The next observation picks out a piece of $B_r(v)$ without excess edges.

\begin{observation}\label{obs:forest-without-excess}\label{obs24}
Suppose $r > K+1$. Then there are integers $r_1$ and $r_2$ satisfying $r \geq r_1 > r_2 \geq 0$ such that $E(B_{r_1}(v))\setminus E(B_{r_2}(v))$ contains no edges from $\Ec$ and $r_1 - r_2 \geq \frac{r}{1 + K} - 1$.
\end{observation}
\begin{proof} 
For every edge~$e$ of~$G$, let $d^+(v,e)$ denote the maximum distance in~$G$ from~$v$ to an endpoint of~$e$.
 Let
$\mathcal{D} = \{ d^+_G(e) \mid  e \in \Ec\}$. Note that $|\mathcal{D}| \leq K$. Thus 
the set $\{1,2,\ldots,\lfloor r\rfloor\} \backslash \mathcal{D}$ can be partitioned into at most 
$K + 1$ intervals of contiguous distances.
Thus, there are integers~$r_1$ and~$r'_2$  with $r_1 \geq r_2'$ such that 
the set $\{r_2', r_2' + 1, \dots, r_1\}$ does not intersect~$\mathcal{D}$ 
 and $r_1 - r_2' + 1 \geq \frac{\lfloor r\rfloor - K}{K + 1} \geq \frac{r - 1 - K}{K + 1} = \frac{r}{K+1} - 1 $.

Let $r_2 = r_2' - 1$. We claim that $(E(B_{r_1}(v))\backslash E(B_{r_2}(v)))\cap\Ec = \emptyset$. To see this, consider any $ e \in \Ec$. 
By the choice of $r_1$ and $ r_2$, either $d^+(v,e)> r_1$ or $ d^+(v,e)\leq r_2$.  In the former case, an endpoint of~$e$ is outside $V(B_{r_1}(v))$, thus $e \notin B_{r_1}(v)$. In the latter case, both endpoints of~$e$ are in $B_{r_2}(v)$, thus $ e\in B_{r_2}(v)$. That concludes the proof.
\end{proof}

\subsection{Definitions/Notation}\label{sec:partnotation}

In addition to a configurations, we will work with \textit{partial configurations}, where some edges are not yet \textit{revealed}. We will use the symbol ``$*$'' to denote that an edge $e$ has not yet been revealed.
\begin{definition}
A \emph{partial configuration} $\A$ is a function $\A : E \rightarrow \{0,1,*\}$. 
We write $R(\A) \coloneqq\A^{-1}(\{0,1\})$ and refer to it as the set of \emph{revealed edges} in~$\A$. We denote by $\In(\A)$ the set of edges $e$ with $\A(e)=1$ (the ``in'' edges)  and by $\Out(\A)$ the set of edges $e$ with $\A(e)=0$ (the ``out'' edges).
We use $\Omega^*$ to denote the set of all partial configurations.
\end{definition}

Note that each configuration is also a partial configuration -- these are precisely the partial configurations~$\A$ with $R(\A)=E$.
Next we define notation for partial configurations.

\begin{definition}
    Let $\A_1,\A_2$ be partial configurations. We say that $\A_2$ is a \emph{refinement} of $\A_1$, written $\A_1\subseteq\A_2$, whenever $\In(\A_1)\subseteq\In(\A_2)$ and $\Out(\A_1)\subseteq\Out(\A_2)$ (and hence $R(\A_1)\subseteq R(\A_2)$ as well). If $R(\A_1)$ and $R(\A_2)$ are disjoint, the \textit{union} of $\A_1$ and $\A_2$, written $\A_1\cup\A_2$, is the partial configuration with $\In(\A_1\cup\A_2) = \In(\A_1)\cup\In(\A_2)$ and $\Out(\A_1\cup\A_2)=\Out(\A_1)\cup\Out(\A_2)$.
\end{definition}

The following distributions on configurations and partial configurations will be useful.

\begin{definition}\label{def:refine}
    Let $\A$ be a partial configuration in~$\Omega^*$.
    Then $\Omega_\A := \{ \F \in \Omega \mid \A \subseteq \F\}$ will denote the set of all configurations that refine $\A$ and $\pi_\A$  is the conditional distribution of~$\pi$ in~$\Omega_\A$, i.e., for $\F \in \Omega_\A$, $\pi_\A(\F) = \pi(\F) /\sum_{\F' \in \Omega_\A} \pi(\F')$.
    Similarly, if $\Omega_\A \cap \Omega^\ord$ is non-empty, then $\pi^\ord_\A$ is the conditional distribution of~$\pi^\ord$ in~$\Omega_\A\cap \Omega^\ord$.  
\end{definition}

\subsection{The coupling process}\label{sec:coupling}

From now on, we assume that $n=|V(G)|$ is sufficiently large, so $r>K + 1$. 
We define $r_1$ and $r_2$ as in Observation~\ref{obs:forest-without-excess}.
Note $r_1\geq \frac{r}{K+1}-1$.
Let $T = T_0[B_{r_1}(v)]$ be the first $r_1 + 1$ levels of the tree $T_0$.
To prove Theorem~\ref{thm:RRwsm}, we construct a coupling of $\F^+\sim\pi_{\B^+_{r_1}(v)}$ and $\F^\ord\sim\pi^\ord$, such that  $\Pr(\F^+(e)\neq\F^\ord(e)) \leq 1/(100 |E|)$.

To do this, we generate two sequences of coupled partial configurations, starting by revealing all edges outside the ball $B_{r_1}(v)$ and progressively revealing the edges of~$T$, starting from its leaves. We reveal edges one sub-tree at a time until either the boundary is contained in a giant component of in edges,
or an edge incident to $B_{r_2}(v)$ is revealed. In the first case, we use Observation~\ref{obs:same-connectivity-boundary-implies-same-marginals} from Appendix~\ref{app:boundary} to couple edges incident to $v$ perfectly. In the other case, we will be able to show the existence of a large polymer (which is an unlikely event).

Formally, we will construct a sequence of edge subsets $F_0 \subseteq F_1 \subseteq \cdots \subseteq E$, sequence of vertex subsets $V_0, V_1, \dots$, and two sequences of partial configurations $\F_0^+\subseteq\F_1^+\subseteq\dots\subseteq\F^+$ and $\F_0^\ord\subseteq\F_1^\ord\subseteq\dots\subseteq\F^\ord$, maintaining the following invariant for all integers~$i\geq 0$:

\[
    R(\F_i^+) = R(\F_i^\ord)= F_i  \mbox{ and }  \In(\F_i^\ord) \subseteq \In(\F_i^+)
\]

For each $i$, all unrevealed edges of some subtree of $T$ get revealed. It will also hold, for each $i$, that $E\setminus E(B_{r_1}(v)) \subseteq F_i$. 
To facilitate this process of revealing the edges, the following definitions will be helpful.
\begin{definition}\label{def:F_revealed}
 Let $F$ be a subset of $E$. Then $\Omega_{F} \coloneqq \{ \A \in \Omega^* \mid R(\A) = F\}$. The distribution~$\pi_F$  on $\Omega_F$ is defined as follows. For $\A\in \Omega_F$, $\pi_F(\A) := \pi(\Omega_\A)$. 
 Similarly, $\pi^\ord_F$ is the distribution on $\Omega_F$ so that  $\pi_F^\ord(\A) = \pi^\ord(\Omega_\A)$. 
\end{definition}
\begin{definition}\label{def:long}
Let $\A\in \Omega^*$ be a partial configuration. Let $F$ be a subset of $E$ with $R(\A) \subseteq F$. Then $\Omega_{\A,F} := \{ \A'\in \Omega_F \mid \A\subseteq \A', R(\A)=F\}$. We define a distribution $\pi_{\A,F}$ on $\Omega_{\A,F}$ given by 
$$\pi_{\A,F}(\A') := \frac{\pi(\Omega_{\A'})}{\sum_{\A'' \in \Omega_{\A,F}} \pi(\Omega_{\A''})} .$$      
\end{definition}

We can now describe the process in detail.
\begin{enumerate}[{Step} 1.]
\item \label{stepone}
Let $i := 0$. Let $F_0 := E \setminus E(B_{r_1}(v))$.
Let $\F_0^\ord\sim\pi^\ord_{F_0}$.
Let $V_0 := \{ w\in V \mid d_G(v,w) = r_1\}$. 

\begin{enumerate}
\item \label{steponea}
If $|\In(\F_0^\ord)| < (1-\eta) |E|$: Generate $\F^\ord \sim \pi^\ord_{\F_0^\ord}$ and generate $\F^+_0\sim\pi_{\B_{r+1}^+(v)}$. Let $\F^+ := \F^+_0$. Terminate (unsuccessfully). \label{step-tree-like:first-terminate}
        
\item \label{steponeb}
Otherwise: Define $\F_0^+$ by $R(\F_0^+) = \In(\F_0^+) = F_0$.  
\end{enumerate}

\item \label{steptwo}
Repeat until $d_G(F_i, v) \leq r_2$:

\begin{description}
\item[\quad] 
Let $W_i$ be the set of all vertices in $\partial_G(F_i)$ 
that are in a component of size less than~$n/2$ in $G[\In(\F_i^\ord)]$.

\begin{enumerate}
\item \label{steptwoa}
If $W_i$ is non-empty: Choose $w_i \in W_i$ 
to maximise $d_G(w_i,v)$.
Let $p_{i}$ be the parent of~$w_i$ in~$T$.
Let $F_{i+1} := F_i \cup E(T_{p_i})$. Generate, optimally coupled, 
$\F^\ord_{i+1}\sim \pi_{\F^\ord_i,F_{i+1}}$ and $\F_{i+1}^+ \sim \pi_{\F^+_i,F_{i+1}}$. Let $V_{i+1} := (V_{i}\backslash V(T_{p_i}))\cup\{p_i\}$. Let $i := i + 1$.  
        
\item \label{steptwob} Otherwise: Generate, optimally coupled, $\F^\ord\sim\pi_{\F_i^\ord}$ and $\F^+\sim\pi_{\F^+_i}$ and terminate (succesfully).
\end{enumerate}
\end{description}

\item \label{stepthree}
Generate, optimally coupled, $\F^\ord\sim\pi_{\F_i^\ord}$ and $\F^+\sim\pi_{\F^+_i}$. Terminate (unsuccesfully).
\end{enumerate}

We next derive some useful observations justifying the definition of the process. We will require the following definition.

\begin{definition}\label{def:boundary}
Let $\A$ be a partial configuration. The \emph{boundary component set} $\xi(\A)$ of $\A$ is the set of equivalence classes corresponding to components of the boundary vertices, i.e.,
$   \xi(\A) = \{ \kappa \cap \partial(R(\A)) \mid 
\mbox{$\kappa$ is a component in  $G[\In(\A)]$}
\}$.
\end{definition}

\begin{observation}\label{obs:inv}
Suppose that the process does not terminate in Step~\ref{steponea} (with $i=0$). Then it maintains the following invariants for all $i\geq 0$ such that $\F_i^\ord, \F_i^+, F_i$ and $V_i$ are defined.

\begin{enumerate}[{Inv}1{(i)}.]
\item \label{inv1} $R(\F_i^\ord) = R(\F_i^+)=F_i$, $\In(\F_i^\ord) \subseteq \In(\F_i^+)$, and $E\setminus E(B_{r_1}(v))\subseteq F_i$, 
\item \label{inv2} For all distinct $u_1, u_2\in V_i$, $V(T_{u_1})\cap V(T_{u_2}) = \emptyset$,
\item \label{inv3} $\bigcup_{u\in V_i} V(T_u) = V(F_i)\cap V(T)$ and $\bigcup_{u\in V_i} E(T_u) \subseteq F_i$,
\item \label{inv4} If $d_G(F_i, v) > r_2$, then $\partial_G (F_i)\subseteq V_i$,
\item \label{inv5} If $F_{i+1}$ is defined, then $F_i \subset F_{i+1}$, and
\item \label{inv6} If $\F_{i+1}^\ord$ and $\F_{i+1}^+$ are defined in the process, then $\F^\ord_i \subseteq \F^\ord_{i+1}$ and $\F^+_i \subseteq \F^+_{i+1}$
\end{enumerate}
\end{observation}
\begin{proof}
    We do an induction on $i$. For the base case let $i = 0$. Inv\ref{inv1}($0$) holds by the definitions of $F_0$, $\F_0^\ord$ and $\F_0^+$.
    Inv\ref{inv2}($0$) follows from the fact that any $u\in V_0$ is a leaf of $T$ and thus $T_u$ contains a single vertex - $u$.
    
The second part of Inv\ref{inv3}($0$) is trivial 
as neither of the subtrees have any edges.  We show the  first part of Inv\ref{inv3}($0$)  by proving that for all $u\in V(T)$, $u\in V_0$ if and only if $u\in V(F_0)$.
First, suppose that $u\in V(F_0)\cap V(T)$.  The definitions of~$F_0$ and~$T$  
imply that $d_G(u,v) = r_1$ so $u\in V_0$, as desired. 
For the other direction,  consider $u\in V_0$.  We need to show that $G$ has an edge 
$\{u,u'\}$ with $d_G(u',v) > r_1$. The vertex~$u$ is a leaf of $T$, hence $u$ has a single incident edge in $E(T)$. Given that $G$ is $\Delta$-regular with $\Delta > 1$, there is some edge $\{u, u'\}$, such that $u'$ is not $u$'s parent in $T$. By Observation~\ref{obs:excess},  $d_G(u', v) \in \{r_1 - 1, r_1, r_1 + 1\}$. By the choice of $r_1$, $ r_2$, and $u'$, $\{u ,u'\} \notin E(B_{r_1}(v))\setminus E(B_{r_2}(v))$, so  $d_G(u', v) = r_1 + 1 > r_1$,  and the result follows.

The invariant Inv\ref{inv4}($0$) follows from the fact that for any $u\in\partial_G (F_0)$ there are edges $\{u ,u_1\} \in F_0$ and $\{u,u_2\}\notin F_0$, so   $d_G(u, v) \leq r_1$ and $d_G(u_2, v) \leq r_1$, and thus also $d_G(u_1, v) > r_1$, so by Observation~\ref{obs:excess} it follows that $d_G(u, v) = r_1$ and thus $u\in V_0$.

For the inductive step, suppose $F_{i+1}$, $\F_{i+1}^\ord$, $F_{i+1}^+$ and $V_{i+1}$ are all defined and suppose that Inv\ref{inv1}($i$)-Inv\ref{inv4}($i$) hold.
We now show Inv\ref{inv5}($i$), Inv\ref{inv6}($i$) and
Inv\ref{inv1}($i+1$)--Inv\ref{inv4}($i+1$).

For Inv\ref{inv5}($i$), $F_{i}\subseteq F_{i+1}$ follows by construction. For the strict inequality, it is enough to show that $\{w_i, p_i\}\in F_{i+1}\setminus F_i$. Since $F_{i+1}$ is defined, $d_G(F_i, v) > r_2$, otherwise the process would terminate before the $(i+1)$th iteration of Step~\ref{stepthree}. Thus by Inv\ref{inv4}($i$) and by $w_i\in \partial_G (F_i)$, $w_i\in V_i$. So 
for any $u\in V_i$, Inv\ref{inv2}($i$) implies that $p_i\notin V(T_u)$. Hence 
Inv\ref{inv3}($i$) implies that $p_i\notin V(F_i)$ so $\{w_i, p_i\}\notin F_i$. But $E(T_{p_i})\subseteq F_{i+1}$, 
hence $\{w_i ,p_i\}\in F_{i+1}$, completing the proof of Inv\ref{inv5}($i$).
Inv\ref{inv6}($i$) follows by construction. 

The first and third parts of Inv\ref{inv1}($i+1$) follow by construction, the second part follows from Corollary~\ref{cor:monotonicity} (see Appendix~\ref{app:boundary}, using $\In(\F_i^\ord)\subseteq\In(\F_i^+)$ and the fact that the coupling is optimal.
Inv\ref{inv2}($i + 1$) follows by construction and Inv\ref{inv2}($i$).
Also, the second part of Inv\ref{inv3}($i + 1$) follows by construction and Inv\ref{inv3}($i$).

For the first part of Inv\ref{inv3}($i+1$), 
by the definition of $F_{i+1}$, if $u\in V(F_{i+1})\cap V(T)$, then $u\in V(T_{p_i})$ or $u\in V(F_{i})\cap V(T)$. In the first case, note $p_i\in V_{i+1}$. Otherwise, by  Inv\ref{inv3}($i$), there is a vertex $u'\in V_i$ with $u\in V(T_{u'})$. If $u'$ is in $V(T_{p_i})$, then so is $u$, hence the first case applies. If not, then $u'\in V_{i+1}$. 
For the converse, for $u\in V_{i+1}$, there are two cases. If $u = p_i$ then
$E(T_{p_i})\subseteq F_{i+1}$ so
$V(T_{p_i}) \subseteq V(F_{i+1})\cap V(T)$. Otherwise, $u\in V_i$, so the result follows by  Inv\ref{inv3}($i$) and Inv\ref{inv5}($i$).

Finally we prove Inv\ref{inv4}($i+1$). Suppose that $d_G(F_{i+1}, v) > r_2$ and let $u\in\partial_G F_{i+1}$. Our goal is to show $u\in V_{i+1}$ Let $u_1, u_2\in V$ be such that $\{u ,u_1\}\in F_{i+1}$ and $\{u ,u_2\}\notin F_{i+1}$. By Inv\ref{inv1}($i+1$), $E\setminus E(B_{r_1}(v))\subseteq F_{i+1}$, and so $u, u_2 \in V(B_{r_1}(v)) = V(T)$. Also,   $d_G(u, v) > r_2$  implies $\{u ,u_2\}\in E(B_{r_1}(v))\setminus E(B_{r_2}(v)$. Since $\{u, u_1\}\in F_{i+1}$, $u\in V(F_{i+1})$, so by Inv\ref{inv3}($i+1$), there is a vertex~$u'\in V_{i+1}$ such that $u\in V(T_{u'})$. By the choice of $r_1$ and $r_2$, $\{u ,u_2\}$ is a tree edge, so if $u\neq u'$, by the second part of Inv\ref{inv3}($i+1$), $\{u ,u_2\}$ would be then in $F_{i+1}$, contradicting the choice of $u_2$. Thus $u = u' \in V_{i+1}$.
\end{proof}

\begin{observation}\label{obs:distributions are correct}\label{obs:dist}
The process eventually terminates, with $\F^\ord \sim\pi^\ord$ and $\F^+\sim\pi_{\B_{r_1}^+(v)}$. 
\end{observation}
\begin{proof}
    By Inv\ref{inv5}, the process always terminates. If the process terminates in Step~\ref{steponea}, then, by construction, the resulting configurations are drawn from the correct distributions. Otherwise, Step~\ref{steptwo} is executed. Thus, $(1-\eta) |E| \leq |\In(\F^\ord_0)|$. By Inv\ref{inv6}, for every value of $i$ until termination, $(1-\eta) |E| \leq |\In(\F^\ord_0)| \leq |\In(\F^\ord_i)|$. This implies that every $\Omega_{\F_i^\ord} \subseteq \Omega^\ord$ so $\pi_{\F_i^\ord}^\ord = \pi_{\F_i^\ord}$. From this it follows that $\F^\ord \sim\pi^\ord$. Given how $\F_0^+$ is generated in Step~\ref{steponeb}, it holds that, for any $\F\in\Omega$, $\pi_{\B_{r_1}^+(v)}(\F) > 0$ precisely when $\F_0^+\subseteq\F$. Then $\F^+\sim\pi_{\B_{r_1}^+(v)}$ follows from the way that $\F^+_i$ is generated in the process.
\end{proof}
 
\begin{observation} \label{obs:if-end-in-Step2b-edges-of-v-agree}
    If the process terminates in Step~\ref{steptwob}, then there is a non-negative integer~$i$ such that $G[\In(\F_i^\ord)]$ has a giant component containing all vertices in~$\partial_G(F_i)$. Also, $\F^\ord$ and $\F^+$ agree on the edges that are incident to~$v$.
\end{observation}
\begin{proof}
The only way that the process can get to Step~\ref{steptwob}, with some index $i$, is if $d_G(F_i, v) > r_2$ and $W_i = \emptyset$. In this case all vertices of the boundary $\partial_G(F_i)$ are contained in the giant component of $G[\In(\F_i^\ord)]$.

Since $r_2 \geq 0$, $d_G(F_i, v) > r_2 \geq 0$, so the edges that are incident to~$v$ are not in~$F_i$. Since $G$ is connected, the boundary $\partial_G (F_i)$ is non-empty.    

Thus $\xi(\F_{i}^\ord)$ consists of a single equivalence class containing all vertices of~$\partial_G(F_i)$. By Inv\ref{inv1}($i$), $\In(\F_i^\ord)\subseteq\In(\F_i^+)$ and $R(\F_i^\ord) = R(\F_i^+)$, 
so  $G[\In(\F_i^+)]$ has a giant component containing all vertices in~$\partial_G(F_i)$ and $\xi(\F_i^+)$ contains a single equivalence class, containing all vetices in~$\partial_G(F_i)$. By Observation~\ref{obs:same-connectivity-boundary-implies-same-marginals}, $\pi_{\F_i^\ord}$ and $\pi_{\F_i^+}$ have the same projection on $E\backslash F_i$, meaning that for any partial configuration $\A\in \Omega_{E\backslash F_i}$, $\pi_{\F_i^\ord}(\F_i^\ord \cup \A) = \pi_{\F_i^+}(\F_i^+ \cup\A)$. Since the coupling in Step~\ref{steptwob} is optimal,  $\F^\ord$ and $\F^+$ agree on the edges incident to~$v$ (since they are not in $F_i$).
\end{proof}

\begin{observation}\label{obs:path-from-w_i-to-r_1-distance}
Suppose that Step~\ref{steptwoa} is executed with index~$i$ and let  $k = r_1 - d_G(w_i, v)$. There are non-negative integers $j_0,\ldots,j_k$
with $j_k=i$ such that, for every
$\ell\in \{0,\ldots,k-1\}$,  
$j_{\ell} < j_{\ell+1}$ and
$w_{j_{\ell+1}} = p_{j_\ell}$.
\end{observation}
\begin{proof}
By construction, $w_i\in\partial_G(F_i)$ and $d_G(F_i,v) > r_2$. By Inv\ref{inv4}($i$), $\partial_G(F_i) \subseteq V_i\subseteq V(T)$. Hence $d_G(w_i, v) \leq r_1$, so $k \geq 0$. The proof is by induction on $k$.
    
For the base case, $k = 0$, it suffices to define $j_0 := i$ and the condition in the statement of the observation is vacuous.

For the induction step, fix $k\geq 0$ and assume that the observation holds whenever Step~\ref{steptwoa} is executed with any index~$j$ such that $r_1 - d_G(w_j, v)\leq k$. Suppose that Step~\ref{steptwoa} is executed with index~$i$ such that $r_1 - d_G(w_i, v) =  k+1$. Define $j_{k+1} := i$. By Invariant~Inv\ref{inv4}($i$), $w_i\in V_i$. By the construction of $V_i$, either $d_G(w_i, v) = r_1$, or there is an index $i'<i$ such that $w_i = p_{i'}$. The first case is ruled out by $d_G(w_i, v) = r_1 - (k+1) < r_1$. In the second case, $w_i$ is the parent of~$w_{i'}$ in the BFS tree, so  $d_G(w_{i'}, v) = d_G(w_i, v) + 1$ and $k= r_1 - d_G(w_{i'},v)$. We finish by applying the inductive hypothesis with index $i'$.
\end{proof}

Equipped with these results, we can now prove the following lemma.
 
\begin{lemma}\label{lemm:agree-on-edges-of-v-or-large-polymer}
Fix $\Delta\geq 5$ and $K, M>0$.  Suppose that $\beta \geq 3 M$.
Suppose that $n$ is sufficiently large so  that $r:= \tfrac\M\beta \log_{\Delta-1} n >K$.
and $|B_{r}(v)| \leq 9 \Delta n/200$. 
Define $r_1$ as in Observation~\ref{obs24} so that
$r_1\geq\frac{r}{K+1}-1$. 
Let $\F^\ord$ and $\F^+$ be generated by the process. Then at least one of the following conditions holds.
\begin{enumerate}
\item \label{condone} $\F^\ord$ and $\F^+$ agree on the edges that are incident to~$v$.
\item \label{condtwo} $|\In(\F^\ord)\setminus E(B_{r_1}(v))| < (1-\eta) |E|$.
\item \label{condthree} $\F^\ord$ contains a polymer of size at least $\frac{r}{400\Delta(1+K)}-1$.
\end{enumerate}
\end{lemma}
\begin{proof}
By Observation~\ref{obs:distributions are correct}, the process eventually terminates. If it terminates in Step~\ref{steponea}, then $|\In(\F^\ord)\setminus E(B_{r_1}(v))| < (1-\eta)|E|$, thus Condition~\ref{condtwo} is satisfied. If the process terminates in Step~\ref{steptwob}, then by Observation~\ref{obs:if-end-in-Step2b-edges-of-v-agree}, Condition~\ref{condone} holds.
    
For the rest of the proof we assume that the process terminates in Step~\ref{stepthree}. 
We will show that Condition~\ref{condthree} holds.
Let $i'$ be the index that triggers the condition $d_G(F_{i'}, v) \leq r_2$ in Step~\ref{steptwo}. Let $i := i' - 1$ so that step~\ref{steptwoa} is run for the last time with index~$i$. By the choice of~$i'$, $d_G(F_i, v) > r_2$. Since $p_{i}$ is an endpoint of an edge in~$F_{i'}$, $d_G(p_i,v) \leq r_2$. Since $w_i\in \partial_G(F_i)$, $d_G(v,w_i) > r_2$. Since  $d_G(w_i, v) = d_G(p_i, v) + 1$, we conclude that $d_G(w_i, v) = r_2 + 1$.  By Observation~\ref{obs:path-from-w_i-to-r_1-distance} with index~$i$ and $k = r_1 - (r_2+1)$, there are non-negative integers $\{j_0,\ldots,j_k\}$ with $j_k=i$ such that, for every $\ell\in \{0,\ldots,k-1\}$, $j_\ell < j_{\ell+1}$ and $p_{j_\ell} = w_{j_{\ell+1}}$. Thus $w_{j_{\ell+1}} \in V_{\ell+1}$.  Note that $w_{j_0},\ldots,w_{j_k}$ is a path in~$T$ and that $w_{j_k}$ is closest to the root~$v$ in this path. By Observation~\ref{obs24} and the definition of~$k$, $k \geq r/(K+1)-2$. Let $r' = \lfloor k / 2\rfloor$,  $J' = \{j_{\ell} \mid r' \leq \ell \leq k\}$ and $W' = \{w_j \mid  j\in J'\}$. Since $V(T) \subseteq V(B_{r+1}(v))$, we have established that for all $w\in W'$ $r_2+1 \leq d_G(v,w) \leq r_1$.  
We now distinguish three cases.
    
\noindent{\bf Case 1.} 
For all $j\in J'$, $w_j$ is not in  a giant component of $G[\In(\F^\ord)]$.
            
Let $F'$ be the set of edges with endpoints in $W'$ and note that $F'$ contains the edges in the path $w_{r'},\ldots,w_{k}$. 
By Observation~\ref{obs:dist}, $\F^\ord \in \Omega^\ord$.
Applying Item~1 of~Lemma \ref{lemm:vertices-in-small-components-have-all-edges-in-single-polymer} to each vertex $u\in W'$, we find that every edge in~$F'$ is contained in a polymer of~$\F^\ord$. Since the vertices in~$W'$ form a path in~$G$, the edges in~$F'$ are all in the same polymer of~$\F^\ord$. Thus, $\F^\ord$ has a polymer of size at least $k/2 \geq  r/(2(K+1))-1$.
        
\noindent{\bf Case 2.} 
There is an index $j \in J'$ such that $w_j$ is in 
the giant component of $G[\In(\F^\ord)]$ 
and the giant component of $G[\In(\F^\ord)\setminus\{w_j, p_j\}]$.
    
We make the following claim. 
    
\noindent{\bf Claim~2a: }{\sl  
There is then a nonempty set~$S$ of at most $9 \Delta |V|/200$ edges of~$G$ such that the size of the component~$\kappa'$ of~$w_j$ in the graph $G[\In(\F^\ord)\backslash S]$ is 
less than $n/2$ and furthermore
$|V(\kappa')| \geq |S|r'$.}

Before proving Claim~2a, we show that it implies Condition~\ref{condthree}, completing the proof of Case~2. 
By Lemma \ref{lemm:vertices-in-small-components-have-all-edges-in-single-polymer}, item~\ref{item:two}, 
all but at most $45\Delta|S|$ vertices of~$\kappa'$ are such that all their incident edges belong to polymer. Thus,
at least $|S|r' - 45\Delta|S|$ vertices of $\kappa'$ are such that all their incident edges are in a polymer.
By Lemma \ref{lemm:vertices-in-small-components-have-all-edges-in-single-polymer}, item~\ref{item:three}, there are 
at most $50\Delta^2|S|$ polymers of $\F^\ord$ containing vertices of~$\kappa'$.
We conclude that there is a polymer of $\F^\ord$ containing at least 
$\frac{|S|r' - 45\Delta|S|}{50\Delta^2|S|}$ vertices of $\kappa'$ with the property that all of their incident edges are in the polymer.
 
Since the number of edges incident to a set of size~$z$ is at least $\Delta z/2$,   
there is a polymer of $\F^\ord$ containing a vertex of $B_{r_1}(v)$ with size at least
$$\frac{\Delta}{2} \times \frac{|S|r' - 45\Delta|S|}{50\Delta^2|S|} \geq \frac{r' - 45\Delta}{100\Delta} \geq \frac{k/2 - 50\Delta}{100\Delta} \geq \frac{r}{400\Delta(1+K)} - 1.$$

To conclude the proof of Case~2, we prove Claim~2a.    
Since the process ends in Step~\ref{stepthree}, Step~\ref{steponea} does not occur, so $|\In(F^\ord_0)|\geq (1-\eta)|E|$. 
Since $\F^{\ord}_0\subseteq \F^\ord_j$ from Inv\ref{inv6},   $\In(F^{\ord}_0)\subseteq \In(F^\ord_j)$, so $|\In(F^\ord_j)|\geq (1-\eta)|E|$. By Corollary \ref{cor:big}, $G[\In(\F^\ord_j)]$ has a (unique) giant component. Call this component $\kappa_j$. Similarly, $G[\In(\F^\ord)]$ has a giant component, call this component $\kappa$. Note that $V(\kappa_j)\subseteq V(\kappa)$. Since $w_j$ was chosen in Step~\ref{steptwoa},   $w_j\in V(\kappa)\backslash V(\kappa_j)$.  
    
Let $\Gamma$ be the set of all paths from $w_j$ to $V(\kappa_j)$ in the graph $G[\In(\F^\ord)]$. Every path $\gamma\in\Gamma$ contains an edge in $E\setminus F_j$. Let $\{u_1^\gamma, u_2^\gamma\}$ be the first such edge in $\gamma$, with $d_G(u_1^\gamma,w_j) < d_G(u_2^\gamma,w_j)$.
Let $S:= \{ \{u_1^\gamma,u_2^\gamma\} \mid \gamma \in \Gamma\}$. 

By the definition of $S$, any path from $w_j$ to $\kappa_j$ goes through an edge in $S$, thus $w_j$ is not connected to $\kappa_j$ in $G[\In(\F^\ord)\setminus S]$. Let $\kappa'$ be the component~of $w_j$ in $G[\In(\F^\ord)\setminus S]$. The component~$\kappa'$ contains no vertices of $\kappa_j$. Since $|V(\kappa_j)|>n/2$, we get $|V(\kappa')| < n/2$.

We conclude by we proving that $|V(\kappa')\cap B_{r_1}(v)|\geq |S|r'$.
The proof follows from a sequence of claims.

\noindent{\bf Claim 2b:}  For any $\gamma \in \Gamma$, $u_1^\gamma \in \partial_G(F_j)$ and $d_G(v,u_1^\gamma) > r_2$. 

The proof of Claim~2b is as follows. If $u_1^\gamma = w_j$, then $u_1^\gamma\in\partial_G(F_j)$ by the choice of $w_j$. Otherwise, 
since $u_1^\gamma$ is the first edge on~$\gamma$ that is not in $E\setminus F_j$, the edges on the path from $w_j$ to~$u_1^\gamma$ are in~$F_j$, so
there is a at least one edge from $F_j$ incident to $u_1^\gamma$, and the edge $\{u_1^\gamma, u_2^\gamma\}\notin F_j$ is also incident to $u_1^\gamma$, so  $u_1^\gamma \in \partial_G(F_j)$. In either case,
$d_G(v,u_1^\gamma) \geq d_G(v,F_j)> r_2$.

\noindent{\bf Claim 2c:} For any $\gamma\in\Gamma$, 
$d_G(v,u_1^\gamma) \leq r_1$, $d_G(v,u_2^\gamma) \leq r_1$,
$u_1^\gamma \in V_j$, and
$u_2^\gamma$ is the parent of $u_1^\gamma$ in $V(T)$

To prove Claim~2c, consider $\gamma \in \Gamma$.
Since the edge $\{u_1^\gamma,u_2^\gamma\} \in E\setminus F_j$, 
both~$u_1^\gamma$ and~$u_2^\gamma$ are in 
$V(B_{r_1}(v))=V(T)$. Since Step~\ref{steptwo} is executed with index~$j$, $d_G(F_j,v) > r_2$. 
So by Claim~2b and Inv\ref{inv4}($j$),   $u_1^\gamma\in V_j$. 
By Invariant~Inv\ref{inv3}($j$), $E(T_{u_1^\gamma})\subseteq F_j$, so the only candidate for
$u_2^\gamma$ is the parent of $u_1^\gamma$ in~$T$.
 
\noindent{\bf Claim 2d:} $|S| =  |\{u_1^\gamma \mid \gamma\in\Gamma\}|$.

By Claim~2c, the mapping of a vertex $u \in \{u_1^\gamma \mid \gamma\in\Gamma\}$ 
to its parental edge in~$T$ is a bijection 
between $ \{u_1^\gamma \mid \gamma\in\Gamma\}$
and~$S= \{ \{u_1^\gamma,u_2^\gamma\} \mid \gamma \in \Gamma\}$.

\noindent{\bf Claim 2e:} For any $\gamma\in\Gamma$, $|V(\kappa')\cap V(T_{u_1^\gamma})| \geq r'$. 

To prove Claim~2e, consider $\gamma \in \Gamma$. We consider two cases. First, suppose $u_1^\gamma\neq w_j$. 
By construction $w_j \in V_j$ and from Claim~2c, $u_1^\gamma\in V_j$. 
By Inv\ref{inv2}($j$), $V(T_{w_j})\cap V(T_{u_1^\gamma}) = \emptyset$.
At the beginning of the proof, we established $r_2+1 \leq d_G(v,w_j) \leq r_1$.
By Claims~2b and~2c, $r_2 + 1 \leq d_G(v,u_1^\gamma) \leq r_1$.
Since the path~$\gamma$ goes from $w_j$ to $u_1^\gamma$ without leaving~$F_j$,
it goes from~$w_j$ to a leaf of $T_{w_j}$ and it finishes by going from a leaf of $T_{u_1^\gamma}$ 
to~$u_1^\gamma$. All of these edges are in~$\kappa'$ and there are at least $r_1 - d_G(v,u_1^\gamma)$ of them. Since, by Claim~2b, $u_1^\gamma \in \partial_G(F_j)$ but, by construction, it is not in~$\kappa_j$,
by the definition of~$W_j$, $u_1^\gamma \in W_j$. By the choice of $w_j$, 
$d_G(u_1^\gamma, v) \leq d_G(w_j, v) \leq r_1 - r'$. Thus $|V(\kappa')\cap V(T_{u_1^\gamma})|\geq r_1 - (r_1 - r') = r'$. 

For the second case, suppose  $u_1^\gamma = w_j$. By the assumption of Case~2, $w_j$ is in a giant component of $G[\In(\F^\ord)\setminus\{w_j, p_j\}]$, thus 
$S$ contains an edge other than the edge 
$\{u_1^{\gamma'},u_2^{\gamma'}\}$ with $\gamma' \in \Gamma$ that is not equal to the edge
$\{w_j,p_j\}$. By Claim~2c, $u_1^{\gamma'} \neq w_j$. Applying the argument from the first case to~$\gamma'$, the path $\gamma'$ starts by going from~$w_j$ to a leaf of $T_{w_j}$ 
so it contains at least $r'$ edges of $T_{w_j}$, all of which are in~$\kappa_j$.

We now use Claims~2d and 2e to finish the proof that $|V(\kappa')| \geq |S| r'$, which completes the proof of Claim~2a, and hence the proof of Case~2.
Consider paths~$\gamma$ and~$\gamma'$ in~$\Gamma$ with $u_1^\gamma \neq u_1^{\gamma'}$.
By Inv\ref{inv2}($j$), $V(T_{u_1^\gamma})\cap V(T_{u_1^{\gamma'}})=\emptyset$, therefore $|V(\kappa')| \geq \sum_{u\in\{u_1^\gamma\mid\gamma\in\Gamma\}} |V(\kappa')\cap V(T_{u_1^\gamma})|$. Claim~2e shows
that each term in the sum  is at least $r'$.
Claim~2d shows that $|S|$ is equal to the number of terms in the sum. Therefore we obtain
$|V(\kappa')| \geq |S| r'$, as required.

\noindent{\bf Case 3.} 
There is an index $j \in J'$ such that $w_j$ is in the giant component of $G[\In(\F^\ord)]$. However, for all $j\in J'$, $w_j$ is not in a giant component of $G[\In(\F^\ord)\setminus\{w_j, p_j\}]$.
    
Let $\kappa$ be the giant component of~$G[\In(\F^\ord)]$. Let $j_\ell = \min\{j \in J' \mid w_j \in V(\kappa)\}$.
    
\noindent{\bf Claim 3: }{\sl  
For each $\ell'$ satisfying $\ell \leq \ell' < k$, $w_{j_{\ell'}}\in V(\kappa)$ and  $\{w_{j_{\ell'}},w_{j_{\ell'+1}}\} \in \In(\F^\ord)$.
} 

We prove Claim~3 by induction on~$\ell'$. The base case is $\ell' = \ell$. In this case, $w_{j_{\ell}}\in V(\kappa)$ is from the definition of~$j_\ell$.
Then since $w_{j_\ell}$ is in a giant component of $G[\In(\F^\ord)]$ but is not in a giant component of $G[\In(\F^\ord)\backslash\{w_{j_\ell}, p_{j_{\ell}}\}]$, the edge $\{w_{j_\ell}, p_{j_{\ell}}\} = \{w_{j_\ell}, w_{j_{\ell+1}}\} \in \In(\F^\ord)$, as required. 
For the induction step, fix $\ell'$ satisfying $\ell \leq \ell' < k-1$ and assume Claim~3 for~$\ell'$.
Since Claim~3 implies $w_{j_{\ell'}}\in\kappa$ and $\{w_{j_{\ell'}}, w_{j_{\ell'+1}}\}\in\In(\F^\ord)$, $w_{j_{\ell'}}$ and $w_{j_{\ell'+1}}$  are in the same component of $G[\In(\F^\ord)]$, namely, $\kappa$.  Then since $w_{j_{\ell'+1}}$ is in a giant component of  $G[\In(\F^\ord)]$ but is not in a giant component of $G[\In(\F^\ord)\backslash \{ w_{j_{\ell'+1}}, p_{j_{\ell'+1}}\}]$, the edge $\{w_{j_{\ell'+1}}, p_{j_{\ell'+1}}\} = \{w_{j_{\ell'+1}}, w_{j_{\ell'+2}}\} \in \In(\F^\ord)$, as required. 
    
Claim~3 implies that for every $j$ satisfying $j_\ell \leq j \leq k$, $w_j\in V(\kappa)$. Let $e = \{w_k,p_k\}$. Since $w_k$ is not in a giant component of $G[\In(\F^\ord)\setminus \{e\}]$ and the edge~$e$ is not on the path $w_{j_\ell},\ldots,w_k$ it follows that no vertices in~$W'$ are in a giant component of $G[\In(\F^\ord)\setminus \{e\}]$. We now consider two cases.
    
\noindent{\bf Case 3a.}  $k - \ell  \geq r'/2 $. 
    
Take $S = \{e\}$. Suppose that $|V|$ is sufficiently large that $|S| < 9\Delta|V|/200$. 
For any $j\in J'$ let $\kappa_j$ be the component containing $w_j$ in $G[\In(\F^\ord)\setminus\{w_j,p_j\}]$. Item~\ref{item:two} of Lemma~\ref{lem:shortname} implies that all but at most~$45\Delta$ vertices of~$\kappa_j$ are such that all their incident edges belong to a polymer of~$\F^\ord$. Item~\ref{item:three} implies that there are at most $50 \Delta^2 $ polymers containing vertices of~$\kappa_j$.
Since $k-\ell\geq r'/2$, we conclude that there are at least $ r'/2 - 45\Delta$ vertices of $W'\cap V(\kappa)$ such that all their incident edges are in a polymer, and at most $50\Delta^2$ polymers containing them. Hence there a polymer of size at least 
$\frac{\Delta}{2} \times\frac{r'/2 - 45\Delta}{50\Delta^2}\geq \frac{r' - 90\Delta}{200\Delta} \geq \frac{r}{400\Delta(1+K)}-1$. 
    
\noindent{\bf Case 3b.} $k - \ell  <  r'/2$.
    
We have already seen that  
$w_{r'},\ldots,w_{\ell-1}$ is a path in~$G$ with length at least $\ell - r' \geq r'/2$ and its vertices are not in $V(\kappa)$. By Item~1 of~Lemma \ref{lemm:vertices-in-small-components-have-all-edges-in-single-polymer}, every edge incident to one of these vertices is contained in a polymer of~$\F^\ord$. Since the vertices form a path in~$G$, the edges adjacent to them are all in the same polymer of~$\F^\ord$. Thus, $\F^\ord$ has a polymer of size at least $(r'/2) \geq \frac{(k/2)-1}{2} \geq \frac{r}{4(K+1)}-1$.
\end{proof}

Next we will show that the termination condition from Step~\ref{stepone} of the process is unlikely to happen.

\begin{lemma}\label{cor:large-q-many-edges-outside-a-ball}\label{cor:eee} 
Let $\Delta\geq 5$ and $K\geq 0$ be integers. Let $\delta\in(0,1/2)$ be a real number. 
Let $\eta = \min\{\delta/5,1/100\}$
There are positive real numbers   $q_0\geq 1$ and $n_0$ such that the following holds for all $q\geq q_0$ and all $\beta \geq \beta_c$.  
Let $G=(V,E)$. Let $n=|V|$ and let $v$ be a vertex in $V$. Let $r$ be a real number satisfying $r \leq\tfrac{1}{3}\log_{\Delta-1}(n)$.   Then 
$\pi_G^\ord(|\In(\F)\backslash E(B_r(v))| < (1-\eta)|E|) =  e^{-\Omega(n)}$.
\end{lemma}
\begin{proof}
Let $r_0(n) = \tfrac{1}{3} \log_{\Delta-1}(n)$ so $r(n) \leq r_0(n)$.
Let $n_0$ be sufficiently large that $\Delta^{r(n_0)+1} \leq (\eta - \zeta) n_0\Delta/2$, where $\zeta$ is the constant from Lemma~\ref{lem:dominant}. By Lemma~\ref{lem:dominant}, 
$$\pi_G^\ord\Big(|\In(\F)|\leq (1-\zeta)|E|\Big)=\emm^{-\Omega(n)}.$$
However, if $|\In(\F)| > (1-\zeta)|E|$,
then \[|\In(\F)\backslash E(B_r(v))| > (1-\zeta)|E| - \Delta^{r(n)+1} \geq
(1-\zeta)|E| - \Delta^{r(n_0)+1}\geq
(1-\eta)|E|.\qedhere\]
\end{proof}

\subsection{Proof of WSM within the ordered phase}

We can now prove Theorem~\ref{thm:RRwsm}.  We start with the following Lemma.
   
\begin{lemma}\label{cor:large-q-wsm-within-the-ordered-phase}\label{lem:better}
Let $\Delta\geq 5$ and  $K\geq 0$ be integers and let $\delta\in(0, 1/2)$ be a real.  There is $M=M(\Delta,K)>0$
such that the following holds for all sufficiently large $q$ and any $\beta \geq \beta_c$.
For all sufficiently large $n$ and any 
$n$-vertex graph $G\in\G_{\Delta,\delta,K}$, the RC model on~$G$ with parameters~$q $ and~$\beta $ has WSM within the ordered phase at radius $r_1$ satisfying $r_1\leq \frac{\M}{\beta}\log_{\Delta-1} (n)$. 
\end{lemma}
\begin{proof}
Recall that $\eta = \min\{\delta/5,1/100\}$. Let $M = 2000 \times 400 \times (2+K) \Delta\log(\Delta-1)$.

Let $r(n,\beta) := \frac{\M}{\beta}\log_{\Delta-1}(n)$.
Let $\beta_0(q) := \log(q^{1.9/\Delta}+1)$.
Let $q_0$ and $n_0$ be large enough that 
\begin{itemize}
\item $\beta_c(q_0) \geq \beta_0(q_0) $.  
\item 
Lemmas~\ref{lemm:large-q-large-polymer-unlikely}, \ref{cor:large-q-many-edges-outside-a-ball} and~\ref{lemm:agree-on-edges-of-v-or-large-polymer} apply 
\item $\beta_0(q_0) \geq 3 M$.
\item $2 n_0^{-3/2} + f(n_0) \leq 1/(50 \Delta n_0)$ where $f(n)$ is the $e^{-\Omega(n)}$ upper bound from Lemma~\ref{cor:large-q-many-edges-outside-a-ball}.
  
\end{itemize}

For every~$\beta$ define $n_1(\beta)$ to be the smallest positive integer such that 
\begin{itemize}
\item    $\frac{r(n_1(\beta),\beta)}{400\Delta(1+K)}-1 \geq \frac{r(n_1(\beta),\beta)}{400\Delta(2+K)}  $
\end{itemize}

Now fix  $q\geq q_0$ and 
$\beta \geq  \beta_c(q) $.  

Consider any $n\geq \max\{n_0,n_1(\beta)\}$.
Consider an $n$-vertex graph $G\in\G_{\Delta,\delta,K}$. Let $r_1$ be as in Lemma~\ref{lemm:agree-on-edges-of-v-or-large-polymer}. 
Use the process to generate $\F^\ord$ and $\F^+$. 
We wish to show that for every edge $e$ incident to~$v$,
$$
\Vert{\pi_{\B_{r_1}^+(v)}(e \mapsto \cdot ) - \pi^{\ord}(e\mapsto \cdot) }\Vert_{\tv} \leq 1/(100|E|) = 1/(50\Delta n)  
$$

By  Observation \ref{obs:distributions are correct} 
$\F^\ord \sim\pi^\ord$ and $\F^+\sim\pi_{\B_{r+1}^+(v)}$. 
By Lemma \ref{lemm:agree-on-edges-of-v-or-large-polymer},
if $\F^\ord$ and $\F^+$ do not agree on an edge incident to $v$, then $\F^\ord$ contains a polymer of size at least $\frac{r(n,\beta)}{400\Delta(1+K)}-1$, or $|\In(\F^\ord)\backslash E(B_{r_1}(v))| < (1-\eta)|E|$. Thus the probability that $\F^\ord$ and $\F^+$ do not agree on an edge incident to~$v$ is at most the sum of probabilities of those two events.

Applying Lemma~\ref{lemm:large-q-large-polymer-unlikely}, noting that 
$\frac{r(n,\beta)}{400\Delta(1+K)}-1 \geq \frac{r(n,\beta)}{400\Delta(2+K)}$
and (given the lower bound on $M$ at the start of the proof) that  this quantity is at least
$2000\log n/\beta$, the probability of the first event is at most $2n^{-3/2}$.

Note that $r(n,\beta) \leq \tfrac13 \log_{\Delta-1}(n)$.
Applying Lemma~\ref{cor:large-q-many-edges-outside-a-ball}, the probability of the second event is 
$f(n) = e^{-\Omega(n)}$. The result follows by summing these probabilities. 
\end{proof}

Finally, we use Lemma~\ref{lem:better} to prove Theorem~\ref{thm:RRwsm}.

\rrwsm*
\begin{proof}
By Lemmas~\ref{lem:treelike} and~\ref{lem:class}, there exist $K=K(\Delta)>0$ and $\delta=\delta(\Delta) \in (0,1/2)$ 
such that, w.h.p.,  $G\sim\G_{\Delta,n}$ is in $\G_{\Delta,\delta,K}$.
Let $M = M(\Delta,K)$ be as in Lemma~\ref{cor:large-q-wsm-within-the-ordered-phase}, and note that $M$ is in fact only dependent on $\Delta$. By Lemma~\ref{cor:large-q-wsm-within-the-ordered-phase}, for all $q$ sufficiently large and $\beta \geq \beta_c$, for all $n$ sufficiently large and any $n$-vertex $G\in \G_{\Delta,\delta,K}$,   $G$ has a WSM within the ordered phase at radius~$r_1$ where
$r_1 \leq \tfrac\M\beta\log_{\Delta-1}(n)$.
\end{proof}

\section{Proof of WSM within the disordered phase} \label{sec:WSMdis}

We will prove that for $q$ large enough, all large enough graphs in $\G_{\Delta,\delta,K}$ have WSM within the disordered phase on all low enough temperatures. We use the same notion of partial configurations as in Section~\ref{sec:partnotation}.

\begin{definition}\label{def:refineDIS}
Let $\A$ be a partial configuration in~$\Omega^*$. If $\Omega_\A \cap \Omega^\dis$ is non-empty, then $\pi^\dis_\A$ is the conditional distribution of~$\pi^\dis$ in~$\Omega_\A\cap \Omega^\dis$.
    
For a set $F\subseteq E$, define the distribution $\pi^\dis_F$ to be the distribution on $\Omega_F$, such that for $\A\in\Omega_F$, $\pi_F^\dis(\A) = \pi^\dis(\Omega_\A)$.
\end{definition}

First we prove the following lemma showing that a suitable coupling exists. 

\begin{lemma}\label{lemm:disordered-coupling}
Fix $\Delta\geq 5$, $K\geq 0$ integers, and $\delta\in(0, 1/2)$ a real. Then, for all $q$ large enough and $\beta\leq \log(q^{2.1/\Delta}+1)$, the following holds for all $G=(V,E)\in\G_{\Delta,\delta,K}$ with sufficiently many vertices.
Let $v\in V$ and $r := \frac{1}{3}\log_{\Delta-1} |V|$. Then there is 
an integer~$r_1$ satisfying $r \geq r_1 \geq \frac{r}{K+1}-1$ and there is a coupling $(\F^-,\F^\dis)$ such that $\F^-\sim\pi_{\B^-_{r_1}(v)}$ and $\F^\dis\sim\pi^\dis$, and moreover at least one of the following holds:
\begin{enumerate}
\item $|\In(\F^\dis\backslash E(B_{r_1}(v))| > \eta|E| - |E(B_{r_1}(v))|$.
\item For any edge $e$ incident to $v$, $\F^-(e) = \F^\dis(e)$. 
\item $G[\In(\F^\dis)]$ has a component with at least $\frac{r}{K+1}-2$ edges.
\end{enumerate}
\end{lemma}
\begin{proof}
We use the definition of $T_0$ and $\Ec$ from Definition~\ref{def:excess-edges-and-tree}.
By Observation~\ref{obs24}, we get $r_1$ and $r_2$ satisfying  $r\geq r_1 > r_2 \geq 0$ such that $r_1-r_2\geq \frac{r}{K+1}-1$ and $E(B_{r_1}(v))\setminus E(B_{r_2}(v))$ contains no edges from $\Ec$. 
Let $T$ be the tree consisting of the first $r_1 + 1$ levels of $T_0$, i.e. $T_0[B_{r_1}(v)]$.
    
Construct the coupling as follows:
First, let $F_0:= E\setminus E(B_{r_1}(v))$ and $\F_0^\dis\sim\pi_{F_0}^\dis$.
If $|\In(\F_0^\dis)| > \eta|E| - |E(B_{r_1}(v))|$, then let $\F^\dis\sim\pi^\dis_{\F_0^\dis}$ and $\F^-\sim\pi^\dis_{\B_{r_1}^-(v)}$.
Otherwise, let $\F_0^-$ be the partial configuration with $R(\F_0^-) = \Out(\F_0^-) = F_0$, let $F_1 = E\setminus E(B_{r_2+1}(v))$ and generate, optimally coupled, $\F_1^\dis\sim\pi_{\F_0^\dis,F_1}$ and $\F_1^-\sim\pi_{\F_0^-,F_1}$. Finally, generate, optimally coupled $\F^-\sim\pi_{\F^-_1}$ and $\F^\dis\sim\pi_{\F^\dis_1}$.

First note, that $\F^\dis\sim\pi^\dis$ and $\F^-\sim\pi_{\B_{r_1}^-(v)}$. In the case $|\In(\F_0^\dis)| > \eta|E| - |E(B_{r_1}(v))|$ it follows by construction. Otherwise, we may note that any refinement of $\F^\dis_0$ has at most $\eta|E|$ edges, as $|E\setminus R(\F_0^\dis)| = |E(B_{r_1}(v))|$, thus $\pi_{\F_0^\dis,F_1} = \pi_{\F_0^\dis,F_1}^\dis$. Similarly, $\pi_{\F^\dis_1} = \pi_{\F_1^\dis}^\dis$. For $\F^-$ the result follows from the fact that a configuration $\F$ is a refinement of $\F_0^-$ if and only if $\pi_{\B^-_{r_1}(v)}(\F)>0$.
    
Now we proceed to show that the resulting configurations satisfy one of the conditions. We have three exhaustive cases, and we show that each corresponds to one of the conditions.
    
\noindent{\bf Case 1.} $|\In(\F^\dis_0)| > \eta|E| - |E(B_{r_1}(v))|$.
    
Then its refinement $\F^\dis$ satisfies $|\In(\F^\dis)\backslash E(B_{r_1}(v))| > \eta|E| - |E(B_{r_1}(v))|$, since $E(B_{r_1}(v))\cap R(\F^\dis_0) = \emptyset$, thus the first condition holds.

\noindent{\bf Case 2.} $|\In(\F^\dis_0)| \leq \eta|E| - |E(B_{r_1}(v))|$, and
$\xi(\F_1^\dis)$ is a free boundary. That is, no two vertices in $\partial(F_1)$ are in the same component of $G[\In(\F_1^\dis)]$. 

Since $|\In(\F^\dis_0)| \leq \eta|E| - |E(B_{r_1}(v))|$, we have $R(\F_0^\dis) = R(\F_0^-) = F_0$, and $R(\F_1^\dis) = R(\F_1^0) = F_1$. Also, $\emptyset = \In(\F_0^-)\subseteq \In(\F_0^\dis)$, thus by Corollary~\ref{cor:monotonicity}, and by the fact that $\F_1^\dis$ and $\F_1^-$ were optimally coupled, also $\In(\F_1^-)\subseteq\In(\F_1^\dis)$. 
So since  no two vertices in $\partial(R(\F_1^\dis)) = \partial(F_1) = \partial(R(\F_1^-))$ are in the same component of $G[\In(\F_0^\dis)]$, it is also true that no two vertices in $\partial(F_1)$ are in the same component of $G[\In(\F_0^-)]$.  So $\xi(\F^\dis_0) = \xi(\F^-_0)$, thus by Observation~\ref{obs:same-connectivity-boundary-implies-same-marginals} and by the fact that they were optimally coupled, $\F^-$ and $\F^\dis$ agree on the edges of $E\setminus F_1 = E(B_{r_2+1}(v))$. Note that since $r_2 + 1 \geq 1$, all edges incident to $v$ are in $E(B_{r_2+1}(v))$, thus $\F^\dis$ and $\F^-$ agree on them. This corresponds to condition two.

\noindent{\bf Case 3.} $|\In(\F^\dis_0)| \leq \eta|E| - |E(B_{r_1}(v))|$ and
there are distinct vertices $u_1, u_2\in\partial(F_1)$ such that $u_1$ and $u_2$ are in the same component of $G[\In(\F_1^\dis)]$.

Consider any path $\gamma$ between $u_1$ and $u_2$ in the $G[\In(\F_1^\dis)]$. Since no edges in $B_{r_2}(v)$ are revealed in $\F_1^\dis$, $\gamma$ is contained in $E\setminus E(B_{r_2}(v))$. Note that since $T_{u_1}$ and $T_{u_2}$ - the subtrees of $T$ rooted in $u_1$ and $u_2$ respectively, are disjoint, $V(\gamma)$ is not contained in $V(T_{u_1})$. Let $e = \{w_1, w_2\}$ be the first edge on $\gamma$ such that $w_1\in V(T_{u_1})$ and $w_2\not\in V(T_{u_1})$. Note that for any vertex $w$ with $r_1 > d_G(w, v) > r_2$, all edges incident to $w$ are edges of $T$, and in particular, if it also holds that if $w$ is a non-leaf vertex in $V(T_{u_1})$, all incident edges to $w$ are in $E(T_{u_1})$.

Then, by construction and the choice of $r_1$ and $ r_2$, $\partial(F_1) = \{u\in V \mid d_G(u, v) = r_2 + 1\}$. Also no edges incident to vertices with distance at most $r_2$ from $v$ were revealed, so it must be the case that $d_G(w_1, v) = r_1$.

Hence there is a path containing at least $r_1 - (r_2 + 1) \geq \frac{r}{K+1} - 2$ edges in $G[\In(\F_1^\dis)]$, and thus also in $G[\In(\F^\dis)]$. So there is a connected component of $G[\In(\F^\dis)]$ with size at least $\frac{r}{K + 1}-2$, which corresponds to the third condition.
\end{proof}

 Next we prove Theorem~\ref{thm:RRwsmB}.

\rrwsmb*
\begin{proof}

 By Lemmas~\ref{lem:treelike} and~\ref{lem:class}, there exist $K=K(\Delta)>0$ and $\delta=\delta(\Delta) \in (0,1/2)$ such that, w.h.p.,  $G\sim\G_{\Delta,n}$ is in $\G_{\Delta,\delta,K}$. Let $\beta_1(q) :=\log(q^{2.1/\Delta}+1)$.
 Let $\zeta$ be the constant from Lemma~\ref{lem:dominant} and let $C'$ be constant implicit in Equation~\ref{eq:margin} so that this equation guarantees
$\pi_G^\dis\Big(|\In(\F)|\geq \zeta|E|\Big) \leq e^{-C' n}$.

Let $q_0$ and $n_0$ be large enough that  

\begin{itemize}
\item $\beta_c(q_0) \leq \beta_1(q_0)$.

\item Lemmas \ref{lem:dominant} and \ref{lemm:disordered-coupling} apply.
Lemma \ref{lemm:polymer-weight-disordered} applies with $C = \frac{1}{3(2+K)\log(\Delta-1)}$.

\item $(\eta - \zeta) n_0 \geq \Delta n_0^{1/3}$ and $C'n_0 \geq \log(100\Delta n_0) $ and  $n_0\geq   (50\Delta)^2  $.
        
\item $\frac{\log_{\Delta-1} n_0}{3(K+1)}-2\geq \frac{\log_{\Delta-1} n_0}{3(K+2)}$,
 
\end{itemize}

Now fix  $q\geq q_0$ and 
$\beta \leq  \beta_c(q) $.  

Consider any $n\geq  n_0$.
Consider an $n$-vertex graph $G=(V,E)\in\G_{\Delta,\delta,K}$. 
Let $r:= \frac13 \log_{\Delta-1}(n)$. Fix $v\in V$.
Let $r_1$ be as in Lemma~\ref{lemm:disordered-coupling}. 
Use the process to generate $\F^\ord$ and $\F^+$. 
We wish to show that for every edge $e$ incident to~$v$,
$$
\Vert{\pi_{\B_{r_1}^-(v)}(e \mapsto \cdot ) - \pi^{\dis}(e\mapsto \cdot) }\Vert_{\tv} \leq 1/(100|E|)   
$$
  
By Lemma~\ref{lemm:disordered-coupling}, 
if $\F^\ord$ and $\F^+$ do not agree on all edges incident to $v$, then 
$|\In(\F^\dis)|\geq \eta|E| - |E(B_{r_1}(v))|$ or
$\F^\ord$ contains a polymer of size at least 
$\frac{r}{K+1}-2\geq \frac{\log n}{3(K+2)\log(\Delta-1)} = C \log n$.
Thus the probability that 
$\F^\ord$ and $\F^+$ do not agree on all edges incident to $v$ is at most the sum of the probabilities of these two events.
 
Since $(\eta - \zeta) n \geq \Delta n^{1/3}$,
$\zeta|E| \leq \eta|E| - |E(B_{r_1}(v))|$, so Lemma~\ref{lem:dominant} guarantees
that the probability of the first event is at most $e^{-C'n}\leq \frac{1}{100\Delta n} = \frac{1}{200|E|}$.

By Lemma~\ref{lemm:polymer-weight-disordered}, the probability of the second event is at most $n^{-3/2} / 2\leq \frac{1}{100\Delta n} = \frac{1}{200|E|}$.
The result follows by summing these probabilities. 
\end{proof}

\section{Proof of Theorems~\ref{thm:main1c} and~\ref{thm:main1d}}\label{sec:thmmain2}
In this section, we 
provide the last ingredients that were used
in the proof of Theorem~\ref{thm:main1} by showing the $O(n\log n)$ bounds in Theorems~\ref{thm:main1c} and~\ref{thm:main1d}. 
The overall argument is quite close to what was presented in Section~\ref{sec:f3434f34} for the proof of Theorem~\ref{thm:main1b};  the only extra argument required is a slightly more refined estimate from the distance to stationarity using log-Sobolev constants to save an $O(\log n)$ factor (a similar argument  appears in \cite{blanca2021random}). We present  first the relevant tools in Sections~\ref{sec:logab} and~\ref{sec:intq}, and then finish the proofs in Section~\ref{sec:comp}.

\subsection{The log-Sobolev constant and mixing time inequality}\label{sec:logab}
Let $\mu$ be a distribution supported on a set $\Omega$. For a function $g:\Omega\rightarrow \mathbb{R}_{\geq 0}$, let \[\Ent_{\mu}[g]:=\Eb_{\mu}[g \log g]-\Eb_{\mu}[g]\log \Eb_{\mu}[\log g],\] 
with the convention $0\log 0:=0$. Moreover, for a reversible Markov chain $(X_t)_{t\geq 0}$ on $\Omega$ with transition matrix $P\in \mathbb{R}^{\Omega\times \Omega}_{\geq 0}$, let $\Ec(g,g):=\tfrac{1}{2}\sum_{\omega,\omega'}\mu(\omega)P(\omega,\omega')(g(\omega)-g(\omega'))^2$ be the Dirichlet form for $P$. The standard log-Sobolev constant of the chain is then defined as 
\[\alpha(P):=\min_{\substack{g:\Omega\rightarrow \mathbb{R}_{\geq 0};\\ \Ent_{\mu}[g]\neq 0}}\frac{\Ec(\sqrt{g},\sqrt{g})}{\Ent_\mu[g]}.\]
The following well-known connection between the log-Sobolev constant and the distance from stationarity can be found, e.g., in~\cite[Fact 6.1]{blanca2021random}. Let $\mu_{\min}=\min_{\omega\in \Omega}. \mu(\omega)$. Then, for any $\gamma<\alpha(P)$, it holds that 
\begin{equation}\label{eq:logSob}
    \mbox{$\max_{X_0}$}\, \mathrm{dist}_{\mathrm{TV}}(X_t,\mu)\leq \emm^{-\gamma t/2}\big(\log\tfrac{1}{\mu_{\min}}\big)^{1/2}.
\end{equation}

\subsection{Log-Sobolev constant for free/wired tree-like neighbourhoods }\label{sec:intq} 
Here we briefly discuss how to bound the log-Sobolev for the wired RC dynamics for tree-like neighbourhoods for integer $q>1$ (cf.  Remark~\ref{rem:treeint}) using the results from \cite{treemixingRC} on the tree.
\begin{lemma}
Let $q\geq 2$ and $\Delta\geq 3$ be  integers, and  $K,\beta>0$  be reals. There exists  $\tilde{C}>0$ such that the following holds for every $\Delta$-regular graph $G=(V,E)$ and any integer $r\geq 1$.

Suppose that $\rho\in V$ is such that $G[B_r(\rho)]$ is $K$-treelike. Then, with $n=|B_r(\rho)|$,  the log-Sobolev constant for the  wired RC dynamics on $B_r(\rho)$ is $\geq \tilde{C}/n$.
\end{lemma}
\begin{proof}

Let $T=T_{\Delta}(\rho)$ denote the $\Delta$-regular tree rooted at $\rho$. Then, for integer $q\geq 2$ and $\beta>0$, it is shown in \cite[Proof of Lemma~32]{treemixingRC} that there exists $C>0$ such that the log-Sobolev constant for the wired RC dynamics on $T[B_r(\rho)]$ is $\geq C/N$ where $N=|B_r(\rho)|$.  

To  translate this into a lower bound on the log-Sobolev constant of $G[B_r(\rho)]$ (which is $K$-treelike) as in the statement of the lemma, we can just use the argument in \cite[Proof of Lemma~4.4]{SinclairsGheissari2022}. There, they show, for the Ising model with all plus boundary condition, it via a graph decomposition argument that there exists a constant $C'=C(q,\beta,K,\Delta)$ independent of $r$ such that  $\alpha\big(P_{G[B_r(\rho)]}\big)\geq C\alpha\big(P_{T[B_r(\rho)]}\big)$ whenever $G[B_r(\rho)]$ is $K$-treelike. The details of the graph decomposition do not depend on the Ising model, so the exact same strategy yields the analogue for the RC model. 
\end{proof}
The analogous result for the free boundary is available from \cite{blanca2021random} for all $q\geq 1$, we state here the following more precise version of Lemma~\ref{lem:mixtreelike}.
\begin{lemma}[{\cite[Lemma 6.5] {blanca2021random}}]\label{lem:bmixtreelike}
Let $\Delta\geq 3$ be an integer, and  $q,K>1$, $\beta>0$ be reals. There exists  $C>0$ such that the following holds for any $\Delta$-regular graph $G$ and integer $r\geq 1$.

Suppose that $\rho\in V$ is such that $G[B_r(\rho)]$ is $K$-treelike. Then, with $n=|B_r(\rho)|$,  the log-Sobolev constant for the  free RC dynamics on $B_r(\rho)$ is $\geq C/n$.
\end{lemma}

\subsection{Completing the proof of Theorems~\ref{thm:main1c} and~\ref{thm:main1d}}\label{sec:comp}

\begin{proof}[Proof of Theorems~\ref{thm:main1c} and ~\ref{thm:main1d}]
We first show Theorem~\ref{thm:main1c}. Consider therefore $\Delta\geq 5$, and let $q$ be sufficiently large so that both Lemma~\ref{lem:dominant} and Theorem~\ref{thm:RRwsmB} apply, and suppose that $\beta \leq \beta_c$.

Consider   $G=(V,E)\sim \G_{n,\Delta}$ with $n=|V|$ and $m=|E|$.  By Lemma~\ref{lem:treelike}, we can  assume that  $G$ is locally $K$-treelike. By Lemma~\ref{lem:dominant},
$\pi_G^\dis\big(|\In(\F)|\geq \zeta |E|\big)=\emm^{-\Omega(n)}$. By Theorem~\ref{thm:RRwsmB}, $G$ has  WSM within the disordered phase at  radius $r$ for some  $r\leq \tfrac{1}{3}\log_{\Delta-1}n$.  

 We will consider the RC dynamics $(X_t)_{t\geq 0}$ with $X_0$ being the all-out configuration on the edges. We will also consider 
the ``disordered'' RC dynamics $(\hat{X}_t)_{t\geq 0}$ with $\hat{X}_0\sim\pi^\dis_G$ where we reject moves of the chain that lead to configurations outside of $\Omega^{\dis}$.  Note that $\hat{X}_t\sim\pi^\dis_G$ for all $t\geq 0$. We will show how to couple these two dynamics in $O(n \log n)$ time. The result will follow by showing that there is a coupling between $(X_t)_{t\geq 0}$ and $(\hat{X}_t)_{t\geq 0}$ such that, for $T=O(n\log n)$, it holds that
\begin{equation}\label{eq:bb4r34mm}
    \Pr(X_T\neq \hat{X}_T)\leq 1/4.
\end{equation}
Analogously to the ordered case we use the monotone coupling, where at every step $t$, the two chains choose the same edge $e_t$ to update and use the same uniform number $U_t\in [0,1]$ to decide whether to include $e_t$ in each of $X_{t+1},\hat{X}_{t+1}$. For $t\geq 0$, let $\Ec_t$ be the event that $\In(\hat{X}_t)\leq \zeta|E|$ and  let $\Ec_{< t}:=\bigcap_{t'=0,\hdots,t-1} \Ec_{t'}$. From Lemma~\ref{lem:dominant} we have that  $\pi^{\dis}_G(\Ec_{< t})\geq 1-t\emm^{-\Omega(n)}$.  Using again the monotonicity of the model for $q\geq 1$, under the monotone coupling, for all $t\geq 0$ such that  $\Ec_{<t}$ holds (and hence no reject move has happened in $\hat{X}_t$ so far), we have that $X_t\leq \hat{X}_t$  (i.e., $\In(X_t)\subseteq \In(\hat{X}_t)$). We next proceed to bound the terms in the upper bound 
\begin{equation*}\tag{\ref{eq:rvffv3}}
\Pr\big(X_{t}\neq \hat{X}_t\big)\leq \sum_{e}\Pr\big(X_t(e)\neq \hat{X}_t(e)\big)\leq m\Pr\big(\overline{ \Ec_{<t}}\big)+\sum_{e}\Pr\big(X_t(e)\neq \hat{X}_t(e)\mid \Ec_{< t}\big).
\end{equation*}
So, fix an arbitrary edge $e$ incident to some vertex $v$, and let $(X^v_t)$ be the free RC dynamics on $G[B_r(v)]$. We couple the evolution of $(X_t^v)$ with that of $(X_t)$ and $(\hat{X}_t)$ using the monotone coupling analogously to the ordered case, where in $X_t^v$ we ignore updates of edges outside the ball $G[B_r(v)]$). We have  $X_t^v\leq X_t$ for all $t\geq 0$, and hence, conditioned on $\Ec_{<t}$, we have that $X^v_t\leq X_t\leq \hat{X}_t$. Proceeding as in the proof of Theorem~\ref{thm:main1b}, we therefore get the following analogue of  \eqref{eq:rvffv2}:
\begin{equation}\label{eq:brvffv2}
\begin{aligned}
\Pr\big(X_t(e)\neq \hat{X}_t(e)\mid \Ec_{< t}\big)\leq4\Pr\big(\overline{ \Ec_{<t}}\big)+\big|\Pr(X^v_t(e)&=0)-\pi_{\B^-_r(v)}(e\mapsto 0)\big|\\
&+\big|\pi_{\B^-_r(v)}(e\mapsto 0)-\pi^{\dis}_G(e\mapsto 0)\big|.
\end{aligned}
\end{equation}
Since $G$ has WSM within the disordered phase at radius $r$, we have that
\begin{equation}\label{eq:brvffv1}
\big|\pi_{\B^-_r(v)}(e\mapsto 0)-\pi^{\dis}_G(e\mapsto 0)\big|\leq 1/(100 m).
\end{equation}
Let $C>0$ be the constant in Lemma~\ref{lem:bmixtreelike} and set  $N_v=|E(B_r(v))|\leq \Delta^{r+1}$. By Chernoff bounds, for $T=\Theta(n\log n)$, we get at least $t_v:=\tfrac{40}{C}N_v\log n$ edge updates withing the ball $B_r(v)$ with probability $1-\exp(-n^{\Omega(1)})$. Since $r\leq \tfrac{1}{3}\log_{\Delta-1} n$, $G[B_r(v)]$ is $K$-treelike, so from Lemma~\ref{lem:bmixtreelike} and applying the log-Sobolev inequality \eqref{eq:logSob} (with $\gamma=C/(2N_v)$, $t=t_v$ and $\min_{\F}\pi_{\B^-_r(v)}(\F)\geq (2q \emm^{\beta})^{-N_v}$), we have
\begin{equation}\label{eq:brvffv}
\big|\Pr(X^v_T(e)=0)-\pi_{\B^-_r(v)}(e\mapsto 0)\big|\leq \exp(-n^{\Omega(1)})+ \emm^{-C t_v/(4N_v) }\big(\log\tfrac{1}{(2q \emm^{\beta})^{-N_v}} \big)^{1/2}\leq 1/m^3,
\end{equation}
where the last inequality holds for all $m=\Delta n=\Omega(1)$.
Plugging \eqref{eq:brvffv1} and \eqref{eq:brvffv} into \eqref{eq:brvffv2} for $t=T$, and then back into \eqref{eq:rvffv3}, we get \eqref{eq:bb4r34mm}, i.e., $\Pr(X_T\neq \hat{X}_T)\leq 5m T \emm^{-\Omega(n)}+m/m^3+1/100\leq 1/5$. This finishes the proof of Theorem~\ref{thm:main1b}.

For the proof of Theorem~\ref{thm:main1d} (integer $q$ and $\beta\geq \beta_c$), the argument is completely analogous to what was just presented for $\beta\leq \beta_c$, using now the log-Sobolev bound of Lemma~\ref{lem:mixtreelikeb} to get the analogue of \eqref{eq:brvffv} (see the proof in Section~\ref{sec:f3434f34} and the inequality \eqref{eq:rvffv}).  
\end{proof}

\bibliographystyle{plainurl}
\bibliography{fn.bib}

\appendix

\section{Mixing on tree-like graphs for the wired RC dynamics}\label{app:mixtreelike}
In this section, we prove Lemma~\ref{lem:mixtreelike} which we restate here for convenience.
\begin{lemmixtreelike}
\statelemmixtreelike
\end{lemmixtreelike}
\begin{proof}
We use a canonical paths argument, following largely \cite[Proof of Lemma 8]{MosselSly}. Let  $G=(V,E)$ be a $\Delta$-regular graph and  $\rho\in V$ be such that $G[B_r(\rho)]$ is $K$-treelike. Let $n=|B_r(\rho)|$, $m=|E(B_r(\rho)|$. Let $\hat{\pi}=\hat{\pi}_{B_r(\rho)}$ be the stationary distribution of the wired RC dynamics on $B_r(\rho)$.

Consider a BFS tree $T$ for $G[B_r(\rho)]$  starting from $\rho$, let $L$ be the set of the leaves of $T$ and let $\Ec=E\backslash E(T)$ be the set of excess edges. Note that the height of $T$ is at most $r$,  $S_r(\rho)\subseteq L$ and there are at most $2K$ leaves of $T$ at depth less than $r$ (by the $\Delta$-regularity of $G$, each such leaf must be incident to an excess edge of $G$), so $|L\backslash S_r(\rho)|\leq 2K$. For a vertex $u$, we let $T_u$ be the subtree of $T$ rooted at $u$, and $\mathrm{Anc}(u)$ be the ancestors of $u$ in $T$ (i.e., the vertices on the path from the root $\rho$ to $u$, including $u$). 

 Consider next a depth-first-search traversal of the tree $T$ starting from $\rho$, and let  $v_1,\hdots, v_n$ be the order in which the vertices of $T$ were first visited. We write $v_i<v_j$ whenever $i<j$; moreover, for a subset of vertices $S$ and a vertex $u$, we write $S<u$ to denote that for each $w\in S$ it holds that $w<u$. Since the tree $T$ has depth $r$, at any stage of the DFS traversal there are at most $r$ vertices which are visited but not fully explored (note that a leaf that is visited is automatically explored). Since the degree of any vertex is $\leq \Delta$, it follows that
\begin{enumerate}[(a)]
\item \label{it:alpha} for $i=1,\hdots,n$, there are at most $\Delta r$ tree edges  with one endpoint in $\{v_1,\hdots,v_i\}$ and the other in $\{v_{i+1},\hdots,v_m\}$ (from the DFS traversal). 
\item \label{it:beta1} for an arbitrary vertex $u$ in $T$ and any vertex $w$ which is not an ancestor of $u$ (i.e., $w\notin \mathrm{Anc}(u)$),  we have from the DFS traversal either that $V(T_w)<u$ or $V(T_w) > u$.
\end{enumerate}
We order the edges $e_1,\ldots,e_m$ 
in $G[\B_r(\rho)]$ (including excess edges)
in lexicographic order in terms of the order of the vertices. So edges with an endpoint~$v_1$ come first (ordered by the other endpoint) and so on.

For two configurations $\F,\F':E\rightarrow \{0,1\}$ that differ on the edges $e_{i_1},\hdots, e_{i_k}$ (with $i_1<\cdots<i_k$), define the path $\mathrm{Path}(\F,\F')$ of configurations $\mathrm{Path}(\F,\F')=\F_0\rightarrow \F_1\cdots\rightarrow \F_{k}$ by defining, for $j=0,\hdots,k$, $\F_{j}$ to agree with $\F$ on the edges $\{e_{i_{j+1}},\hdots, e_{i_k}\}$ and to agree with $\F'$ on the edges $\{e_{i_1},\hdots, e_{i_j}\}$. Note that $\F_0=\F,\F_k=\F'$ and both $\F,\F'$ and hence $\F_j$ as well agree on $E\backslash \{e_{i_1},\hdots, e_{i_k}\})$. Then, the relaxation time of the chain, see for example \cite[Chapter 5]{Jerrumbook}, is bounded by 
\[\tau\leq m\max_{(\F,\F')}\sum_{\substack{\A,\A':E\rightarrow \{0,1\};\\ (\F\rightarrow \F')\in \mathrm{Path}(\A,\A')}}\frac{\hat\pi(\A)\hat\pi(\A')}{\hat\pi(\F)\Pr(\F,\F')},\]
where  the maximum is over all pairs of configurations $\F,\F'$ that differ on a single edge and the summation is over all pairs of configurations $\A,\A'$ whose  corresponding $\mathrm{Path}(\A,\A')$ includes the transition $\F\rightarrow \F'$ as part of the path. 

Consider an arbitrary  pair of configurations $\kappa=(\F,\F')$ that differ at a single edge $f=\{v_j, v_{j'}\}$ for some $j<j'$ and suppose further that $f=e_\ell$.
For configurations  $\A,\A'$ such that $(\F\rightarrow \F')\in \mathrm{Path}(\A,\A')$ define the configuration $g_{\kappa}(\A,\A')$ so that
\begin{equation}\label{eq:agree332r}
\mbox{$g_{\kappa}(\A,\A')$ agrees with $\A$ on $\{e_1,\hdots,e_{\ell-1}\}$ and with $\A'$ on  $\{e_\ell,\hdots,e_{m}\}$.}
\end{equation}From the fact that $(\F\rightarrow \F')\in \mathrm{Path}(\A,\A')$, we also have that:
\begin{equation}\label{eq:agree21341}
\mbox{$\A$ agrees with $\F$ on $\{e_{\ell},\hdots,e_{m}\}$},\ \ \mbox{$\A'$ agrees with $\F'$ on $\{e_{1},\hdots,e_{\ell-1}\}$.}
\end{equation}
It follows that the map $g_{\kappa}(\cdot,\cdot)$ is injective, i.e., given its value $\F^*$ and the configurations $\F,\F'$, there is a unique pair $(\A,\A')$ such that $\F^*=g_{\kappa}(\A,\A')$ and $(\F\rightarrow \F')\in \mathrm{Path}(\A,\A')$. 
For a configuration $\Xc: E(B_r(\rho))\rightarrow \{0,1\}$, let $\hat c(\Xc)$ denote the number of components in $(B_r(\rho),\In(X))$ that do not include any of the vertices in $S_r(\rho)$, cf. Footnote~\ref{fn:chat}. 
Let $s:= \Delta r + K$ and $t := 4 \Delta r + 12 K$.
The main step in the proof is to show that:
\begin{gather}
    |\In(\A)|+|\In(\A')|-|\In(\F)|-|\In(g_\kappa(\A,\A'))|\leq   s\label{eq:edges123},\\
    |\hat c(\A)+\hat c(\A')-\hat c(\F)-\hat c(g_{\kappa}(\A,\A'))|\leq  t.\label{eq:comp123}
\end{gather}
Assuming these for the moment, we have that 
\[
\frac{\hat\pi(\A)\hat\pi(\A')}{\hat\pi(\F)\hat\pi\big(g_{\kappa}(\A,\A')\big)}\leq q^{t}\emm^{\beta s}.\]

Recall that $p=1-\emm^{-\beta}$
and that for $q>1$ it holds that $\hat p\in(p/q,p)$. Therefore,
for the RC dynamics, we have that $\Pr(\F,\F')\geq 
\frac{\min\{p/q,(1-p)\}}{m} \geq (1-\emm^{-\beta})/(q m \emm^\beta)$ using that $\min\{p/q,(1-p)\}\geq \emm^{-\beta}/q$. 
Using this and the injectivity of $g$, we can bound $\tau$ by
\[\tau\leq \frac{m^2 q^{t+1}\emm^{\beta (s+1)}}
{(1-\emm^{-\beta})}
\sum_{\substack{\A,\A':E\rightarrow \{0,1\};\\ (\F\rightarrow \F')\in \mathrm{Path}(\A,\A')}} \hat\pi\big(g_{\kappa}(\A,\A')\big)\leq 
\frac{m^2 q^{t+1}\emm^{\beta (s+1)}}
{(1-\emm^{-\beta})}.\]
The mixing time is at most
$ \tau (1 + \tfrac12 \log(\min_{\Xc} \hat\pi(\Xc)^{-1}))$
Since 
$\min_{\Xc}\hat\pi(\Xc)\geq 
2^{-n} q^{-n} e^{-\beta m}$
and $m\leq \Delta n$, 
the mixing time is at most
 $n^3 \Delta^2 q^{t+1}\emm^{\beta (s+1)}\tfrac{1+ \log(q) + \beta \Delta}{1-\emm^{-\beta}}$
so
the statement for the mixing time follows by taking 
$ \hat{C}:= 2 \Delta^3 (q \emm^{\beta})^{12K+3}/{(1-\emm^{-\beta})}$.

It remains to show \eqref{eq:edges123} and \eqref{eq:comp123}. For convenience let $\F^*=g_{\kappa}(\A,\A')$. 
We will shortly prove the following facts for an arbitrary edge $e=\{u,w\}$:
\begin{enumerate}
\item\label{p:one} if $\A(e)=\A'(e)$, then $\A(e)=\A'(e)=\F(e)=\F^*(e)$.
\item \label{p:two} if $\A(e)\neq \A'(e)$, then $\big(\A(e), \A'(e)\big)= \big(\F^*(e), \F(e)\big)$ when 
$u,w\leq v_j$ 
and $\big(\A(e), \A'(e)\big)= \big(\F(e), \F^*(e)\big)$ when 
$u,w>v_j$. 
\end{enumerate}
Indeed, for Item~\ref{p:one}, consider 
$e$ with $\A(e) = \A'(e)$
and suppose first that  $ \A(e)=\A'(e)=1$. Then $\F(e)=\F'(e)=1$ (since by the construction of $\mathrm{Path}(\A,\A')$, all configurations in it agree on the edges where $\A,\A'$ agree). Moreover,  $\F^*(e)=1$ by construction (since $\F^*=g_{\kappa}(\A,\A')$ agrees with $\A,\A'$ on the edges where the latter agree). The proof for the second case  $ \A(e)=\A'(e)=0$ is analogous.  Item~\ref{p:two} follows from \eqref{eq:agree332r} and \eqref{eq:agree21341} after observing that $u,w\leq v_j$ implies that 
$e=\{u,w\}<e_\ell$, while $u,w> v_j$  gives that $e>e_\ell$.

Let $W_j$ be the tree edges which are incident to a vertex in $\mathrm{Anc}(v_j)$, so that $|W_j|\leq \Delta r$. 
 From Item~\eqref{it:alpha}, these are the only tree edges that can have one endpoint in $\{v_1,\hdots, v_j\}$ and the other among $\{v_{j+1},\hdots, v_n\}$, along with any non-tree edges (i.e., excess edges $\Ec$). Therefore, from Items~\ref{p:one} and~\ref{p:two} above, it follows that for an edge $e\in E\backslash (W_j \cup \Ec)$ it holds that
\begin{equation}\label{eq:44412}
\A(e)+\A'(e)=\F(e)+\F^*(e),
\end{equation}
which establishes \eqref{eq:edges123} since $|W_j \cup \Ec|\leq \Delta r+K$.

To show \eqref{eq:comp123}, fix an arbitrary root-to-leaf path $P$ passing through $v_j$, and denote by $E_P$ be the edges of the path.  Now, consider the slightly ``tweaked'' configurations $\hat{\A},\hat{\A}',\hat{\F},\hat{\F}^*$ obtained from $\A,\A',\F,\F^*$ respectively by setting, for each $\Xc\in \{\A,\A',\F,\F^*\}$,  $\hat{\Xc}(e)=1$  for $e\in E_P\cup W_j$, $\hat{\Xc}(e)=0$  for $e\in \Ec$ and $\hat{\Xc}(e)=\Xc(e)$ for $e\notin E_P\cup W_j\cup \Ec$. Note that, for each $\Xc\in \{\A,\A',\F,\F^*\}$,  we have $\big||\hat{\Xc}|-|\Xc|\big|\leq \Delta r+K$, so $\big||\hat c(\hat\Xc)|-|\hat c(\Xc)|\big|\leq \Delta r+K$ as well. So, to prove \eqref{eq:comp123}, it suffices to show
\begin{equation}\label{eq:fgt5}
|\hat c(\hat\A)+\hat c(\hat\A')-\hat c(\hat\F)-\hat c(\hat\F^*)|\leq 8K.
\end{equation}
To show this, first note that $\hat{\A},\hat{\A}',\hat{\F},\hat{\F}^*$ still satisfy Items~\ref{p:one} and \ref{p:two} above (replacing $\A$ with $\hat\A$, and so on) since the only changes are for edges in  $E_P\cup W_j\cup \Ec$ on which the assignments of the tweaked configurations are identical (and hence fall under Item~\ref{p:one}). In fact, Items~\ref{p:one} and~\ref{p:two} now apply to all edges for the tweaked configurations (previously, edges in $W_j\cup \Ec$ were potentially not covered in Item~\ref{p:two}, but now are covered since they fall under Item~\ref{p:one}). So, the analogue of \eqref{eq:44412}  for $\hat{\A},\hat{\A}',\hat{\F},\hat{\F}^*$ holds for every edge of $G[B_r(\rho)]$, yielding that
\begin{equation}\label{eq:43t34}
|\In(\hat\A)|+|\In(\hat \A')|=|\In(\hat\F)|+|\In(\hat \F^*)|,
\end{equation}
and, further, for any edge $e=\{u,w\}$ we have that 
\begin{equation}\label{eq:hatA}\big(\hat\A(e),\hat\A'(e)\big)=\begin{cases} \big(\hat\F^*(e),\hat\F(e)\big) & \mbox{ if } u,w\leq v_j\\ \big(\hat\F(e),\hat\F^*(e)\big),& \mbox{ otherwise. } \end{cases}
\end{equation}

At this stage, it will be convenient to consider the graph $T^*$ obtained from $G[B_r(\rho)]$ by identifying all the leaves of $T$ into a single vertex $v^*$ (recall that the leaves of $T$ consist of the vertices $S_r(\rho)$, which are at depth $r$, together with at most $2K$ leaves that are at depth $<r$). For $\Xc\in \{\hat\A,\hat\A',\hat\F,\hat\F^*\}$, let $c^*(\Xc)$ be the number of components in the graph $T^*[\In(\Xc)]$ that do not include $v^*$ and let $\C^*(\Xc)=(V^*(\Xc),E^*(\Xc))$ be the connected component of $v^*$. Note that any connected component of $(B_r(\rho),\In(\Xc))$ that does not include a vertex from $S_r(\rho)$ is not connected to $v^*$ in $T^*[\In(\Xc)]$, unless it contains one of the $2K$ leaves at depth $<r$. It follows that  $|c^*(\Xc)-\hat{c}(\Xc)|\leq 2K$, so to prove \eqref{eq:fgt5} it suffices to show
\begin{equation}\label{eq:fgt5b}
c^*(\hat\A)+c^*(\hat\A')= c^*(\hat\F)+ c^*(\hat\F^*).
\end{equation}
 For $\Xc\in \{\hat\A,\hat\A',\hat\F,\hat\F^*\}$, note that  a  connected component $F$ of $\Xc$ that does not include $v^*$ is a subgraph of $T$  (using that $\Xc(e)=0$ for $e\in \Ec$) and hence a tree, so $1=|V(F)|-|E(F)|$. Summing over all such components, we obtain that
 \[c^*(\Xc)=|V(T^*)|-|V^*(\Xc)|-|\In(\Xc)|+|E^*(\Xc)|.\]
Therefore, using \eqref{eq:43t34}, to finish the proof of \eqref{eq:fgt5b} it suffices to  show that for an arbitrary vertex $v$ and an arbitrary edge $e$ it holds that 
\begin{gather}
\mathbf{1}\{v\in V^*(\hat\A)\}+\mathbf{1}\{v\in V^*(\hat\A')\}=\mathbf{1}\{v\in V^*(\hat\F)\}+\mathbf{1}\{v\in V^*(\hat\F^*)\}.\label{eq:indicatorsb}\\
\mathbf{1}\{e\in E^*(\hat\A)\}+\mathbf{1}\{e\in E^*(\hat\A')\}=\mathbf{1}\{e\in E^*(\hat\F)\}+\mathbf{1}\{e\in E^*(\hat\F^*)\}.\label{eq:indicatorsc}
\end{gather}
Clearly \eqref{eq:indicatorsb} is true for any 
$v\in \mathrm{Anc}(v_j)$ 
since $v\in \C(\Xc)$ for each $\Xc\in \{\hat\A,\hat\A',\hat\F,\hat\F^*\}$ by construction of the tweaked configurations. So, consider 
$v\notin \mathrm{Anc}(v_j)$.
First, suppose $v<v_j$.
Let $u$ be the ancestor of $v$ whose parent $w$ is an ancestor of $v_j$.   The key point is that for each $\Xc\in \{\hat\A,\hat\A',\hat\F,\hat\F^*\}$ we have  $w\in V^*(\Xc)$ (since $w$ is an ancestor of $v_j$) and the edge $e=\{u,w\}$ belongs to $W_j$ , so $\hat\Xc(e)=1$. Therefore whether  
$v \in V^*(\Xc)$
is determined by the set of those paths contained in the subtree $T_u$ which  start from $v$ and end either at  $u$  or a leaf of $T_u$. Denote by $\Pc_{v}$ the set of all such paths. 
Since 
$v<v_j$, we  have that $u<v_j$  and that $u$ is not ancestor of $v_j$, so from Item~\ref{it:beta1} it holds that $V(T_u)<v_j$. It follows from~\eqref{eq:hatA} that for every edge $e\in E(T_u)$ we have $(\hat\A(e),\hat\A'(e))=(\F^*(e),\hat\F^*(e))$, so
\[v\in V^*(\hat\A) \Leftrightarrow\, \mbox{$\exists P'\in \Pc_{v}:\ E(P')\subseteq \In(\hat\A)$} \Leftrightarrow\,\mbox{$\exists P'\in \Pc_{v}:\ E(P')\subseteq \In(\hat\F^*)$} \Leftrightarrow  v\in V^*(\hat\F^*),\]
and similarly $ v\in V^*(\hat\A')\Leftrightarrow v\in V^*(\hat\F)$, proving \eqref{eq:indicatorsb} for $v<v_j$. An analogous argument applies when $v>v_j$, then $V(T_u)>v_j$ and, using~\eqref{eq:hatA} again, $v\in V^*(\hat\A)\Leftrightarrow v\in V^*(\hat\F)$, $ v\in V^*(\hat\A')\Leftrightarrow v\in V^*(\hat\F^*)$, finishing the proof of \eqref{eq:indicatorsb}. The proof of \eqref{eq:indicatorsc} is very similar, after noting that for $\Xc\in \{\hat\A,\hat\A',\hat\F,\hat\F^*\}$ and an edge $e=\{u,w\}$ it holds that $e\in E^*(\Xc)$ iff $ \Xc(e)=1$ and $u,w\in V^*(\Xc)$.

This completes the proof of \eqref{eq:fgt5b}, and hence the proof of \eqref{eq:fgt5} and \eqref{eq:comp123} as well, concluding therefore the lemma.
\end{proof}

\section{Boundary conditions and monotonicity}
\label{app:boundary} 

We will be interested in vertices on the boundary of a partial configuration, and in the structure of connected components in the graph induced by in-edges. See Definition~\ref{def:boundary}.

\begin{observation}\label{obs:same-connectivity-boundary-implies-same-marginals}\label{obs:shortname}
Let  $\F_1$ and $\F_2$ be two partial configurations such that $R(\F_1) = R(\F_2)$ and $\xi(\F_1)=\xi(\F_2)$. Let $F = E\backslash R(\F_1)$. Then $\pi_{\F_1}$ and $\pi_{\F_2}$ have the same projection on $F$, meaning that for any partial configuration  $\A\in \Omega_F$, $\pi_{\F_1}(\F_1 \cup \A) = \pi_{\F_2}(\F_2 \cup\A)$.
\end{observation}
\begin{proof}
    Let $\A$ be a partial configuration in~$\Omega_F$. Consider the components of
    $G[\In(\F_1\cup\A)]$ and $G[\In(\F_2\cup\A)]$. 
    \begin{itemize}
        \item 
            The components of~$G[\In(\F_1\cup \A)]$ containing only vertices of $V_R(\F_1)\backslash\partial\F_1$ do not depend on the partial configuration~$\A$. Let $C_1$ be the number of these components. Similarly, the components of~$G[\In(\F_2 \cup \A)]$ containing only vertices of $V_R(\F_2) \backslash \partial \F_2$ do not depend on~$\A$. Let $C_2$ be the number of these components.  
        \item 
            The components 
            of~$G[\In(\F_1\cup \A)]$
            containing only vertices of $V\backslash V_R(\F_1)$ depend on~$\A$ but not on~$F$,
            so the same components are in~$G[\In(\F_2\cup \A)]$.  
        \item 
            Since $\xi(\F_1) = \xi(\F_2)$, there is one-to-one correspondence between components of~$G[\In(\F_1 \cup \A)]$ that contain vertices in $\partial \F_1$ and components of~$G[\In(\F_2 \cup \A)]$ containing vertices in $\partial \F_2$. Any two components that correspond to each other induce the same component on~$F$.
    \end{itemize}
    
    Thus, for any $\A \colon F \rightarrow \{0,1,*\}$, 
    the number of components in $G[\In(\A \cup \F_1)]$ minus~$C_1$ is equal to the number of components in $G[\In(\A \cup \F_2)]$ minus~$C_2$.  
    Hence $w_G(\A\cup\F_1)$ and $w_G(\A\cup\F_2)$ differ by a constant multiplicative factor, from which the result follows.
\end{proof}

Cor~\ref{cor:monotonicity} gives information about
marginals in the case where $R(\F_1) = R(\F_2)$ and $\In(\F_1) \subseteq \In(\F_2)$. The proof follows from Observation~\ref{obs:stuff}.

\begin{observation}\label{obs:stuff}
    Let $f$ be an edge in $E$. Consider partial configurations $\F_1$ and $\F_2$ with $R(\F_1) = R(\F_2)= E\backslash \{f\}$ and $\In(\F_1) \subseteq \In(\F_2)$. Let $\A_1 \sim \pi_{\F_1}$ and $\A_2 \sim \pi_{\F_2}$. Then $\Pr( f\in \In(\A_1)) \leq \Pr(f\in \In(\A_2))$.
\end{observation}
\begin{proof}
    First, observe that if $G[\In(\F_1)]$ has two components connected by~$f$ then the vertices in these components are connected in $G[\In(\F_2)\cup \{f\}]$.
    
    For $j\in\{0,1\}$, let $\F_{i,j}$ be $\F_i\cup\{f\mapsto j\}$. Then, by the observation, $c(\F_{2,0})-c(\F_{2,1})\leq c(\F_{1,0})-c(\F_{1,1})$. It follows that
        
    \begin{align*}
    \Pr(f\in\In(\A_1)) &= \frac{\Pr(\F_{1,1})}{\Pr(\F_{1,1})+\Pr(\F_{1,0})} = 
    \frac{(e^\beta-1)^{|\F_{1}^{-1}(1)|+1} q^{c(\F_{1,1})}}{
    (e^\beta-1)^{|\F_{1}^{-1}(1)|+1} q^{c(\F_{1,1})} + 
    (e^\beta-1)^{|\F_{1}^{-1}(1)|} q^{c(\F_{1,0})}} \\
    &= \frac{(e^\beta-1)q^{c(\F_{1,1})-c(\F_{1,0})}}{(e^\beta-1)q^{c(\F_{1,1})-c(\F_{1,0})} + 1} \leq
    \frac{(e^\beta-1)q^{c(\F_{2,1})-c(\F_{2,0})}}{(e^\beta-1)q^{c(\F_{2,1})-c(\F_{2,0})} + 1} \\
    &= \Pr(f\in\In(\A_2))
    \end{align*}
        
    where the inequality follows from $c(\F_{1,1})-c(\F_{1,0})\leq c(\F_{2,1})-c(\F_{2,0})$.
\end{proof}

In the language of~\cite{GrimmettGeoffrey2006TrmRCM}, Observation~\ref{obs:stuff} says that $\pi$ is $1$-monotonic. It implies the following monotonicity result.

\begin{corollary}\label{cor:monotonicity}
    Let $f$ be an edge of $E$.
    Consider partial configurations $\F_1$ and $\F_2$ with $R(\F_1) = R(\F_2)$ and $\In(\F_1) \subseteq \In(\F_2)$. Let $\A_1 \sim \pi_{\F_1}$ and $\A_2 \sim \pi_{\F_2}$. Then $\Pr( f\in \In(\A_1)) \leq \Pr(f\in \In(\A_2))$. 
\end{corollary}
\begin{proof}
    Use \cite[Theorem 2.27]{GrimmettGeoffrey2006TrmRCM}. It says that if a RCM measure is $1$-monotonic, then the measure is monotonic, i.e. $\pi_{\F_1}(A) \leq \pi_{\F_2}(A)$ for any increasing event -- that is any $A\subseteq\Omega$ such that whenever $\F_1', \F_2'$ are configurations with $\In(\F_1')\subseteq\In(\F_2')$ and $\F_1'\in A$, then also $\F_2'\in A$.
    
    Observation~\ref{obs:stuff} says that $\pi$ is $1$-monotonic, thus monotonicity follows by the theorem.    
    To conclude, take $A$ to be the increasing event ``$f$ is occupied'' -- we can see that this is increasing as whenever, for configurations $\F_1', \F_2'$, $f\in\In(\F_1')$, and $\In(\F_1')\subseteq\In(\F_2')$, then also $f\in\In(\F_2')$.
\end{proof}

\end{document}